\RequirePackage[l2tabu, orthodox]{nag}
\documentclass[a4paper, oneside, 11pt]{amsart}

\usepackage[utf8]{inputenc}
\usepackage[T1]{fontenc}

\usepackage{amsxtra,setspace,xspace,graphicx,lmodern,psfrag,latexsym,bbm,comment,mathtools,amsthm, enumerate,enumitem,pifont,caption,microtype,dsfont,amssymb}
\usepackage[dvipsnames]{xcolor}
\mathtoolsset{showonlyrefs,showmanualtags}
\usepackage[hidelinks,colorlinks=true]{hyperref}
\hypersetup{urlcolor=blue, citecolor=blue ,linkcolor=blue}

\usepackage{theoremref}
\usepackage[a4paper]{geometry}

\makeatletter
 \def\@textbottom{\vskip \z@ \@plus 17pt}
 \let\@texttop\relax
\makeatother

\newtheorem{thm}{Theorem}[section]
\newtheorem{cor}[thm]{Corollary}
\newtheorem{lem}[thm]{Lemma}
\newtheorem{prop}[thm]{Proposition}

\theoremstyle{definition}
\newtheorem{defn}[thm]{Definition}
\theoremstyle{remark}
\newtheorem{rem}[thm]{Remark} 
\theoremstyle{definition}
\newtheorem{assumption}[thm]{Assumption}

\numberwithin{equation}{section}
\newenvironment{tenumerate}[1][]
  {\enumerate[label=\textup(\alph*\textup),ref=(\alph*),#1]}
  {\endenumerate}

\DeclarePairedDelimiter{\abs}{\lvert}{\rvert}
\DeclarePairedDelimiter{\norm}{\lVert}{\rVert}
\DeclarePairedDelimiter{\bra}{(}{)}
\DeclarePairedDelimiter{\pra}{[}{]}
\DeclarePairedDelimiter{\set}{\{}{\}}
\DeclarePairedDelimiter{\skp}{\langle}{\rangle}

\DeclareMathAlphabet{\mathup}{OT1}{\familydefault}{m}{n}
\newcommand{\dx}[1]{\mathop{}\!\mathup{d} #1}
\newcommand{\vol}{\mathup{vol}}

\newcommand{\cC}{\ensuremath{\mathcal C}}

\newcommand{\cG}{\ensuremath{\mathcal G}}

\newcommand{\cI}{\ensuremath{\mathcal I}}

\newcommand{\cM}{\ensuremath{\mathcal M}}

\newcommand{\cP}{\ensuremath{\mathcal P}}

\newcommand{\cX}{\ensuremath{\mathcal X}}

\newcommand{\Leb}{\ensuremath{{L}}}
\newcommand{\SobH}{\ensuremath{{\dot H}}}

\newcommand{\N}{{\mathbb N}}
\newcommand{\R}{{\mathbb R}}

\newcommand{\T}{{\mathds T}^d_L}
\newcommand{\Z}{{\mathds Z}}

\newcommand{\bP}{{\mathbb P}}

\newcommand{\eps}{{\varepsilon}}

\newcommand{\intT}[1]{\int_{\T} #1 \dx{x}}

\DeclareMathOperator{\supp}{supp}
\DeclareMathOperator{\Lip}{Lip}

\DeclareMathOperator{\epi}{epi}

\renewcommand{\tilde}{\widetilde}
\renewcommand{\bar}{\overline}
\renewcommand{\eps}{\varepsilon}
\allowdisplaybreaks[4] 

\title[A mountain pass theorem for the McKean--Vlasov energy]{Barriers of the McKean--Vlasov energy via 
a mountain pass theorem in the space of probability measures}
\author{Rishabh S. Gvalani}
\address{Department of Mathematics, Imperial College London, London SW7 2AZ}
\email{rishabh.gvalani14@imperial.ac.uk}
\author{Andr\'e  Schlichting}
\address{Institut f\"ur Angewandte Mathematik, Universit\"at Bonn}
\email{schlichting@iam.uni-bonn.de}
\thanks{RSG is funded by an Imperial College President's PhD Scholarship, partially through EPSRC Award Ref. 1676118.
AS is supported by the Deutsche Forschungsgemeinschaft (DFG, German Research Foundation) under Germany's Excellence Strategy EXC 2047
-- 390685813, the \emph{Hausdorff Center for Mathematics}, as well as the Collaborative Research Center 1060 -- 211504053, \emph{The Mathematics of Emergent Effects} at the Universität Bonn. 
RSG also acknowledges the hospitality of RWTH Aachen University. AS acknowledges the hospitality of Imperial College London. 
Part of this work was carried out at the workshop {\emph ``Nonlocal differential equations in collective behaviour''} held at the American Institute of Mathematics, San Jos\'e and at the {\emph ``Junior Trimester Programme in Kinetic Theory''} held at the Hausdorff Research Institute for Mathematics, Bonn. RSG and AS are grateful to both institutes for their hospitality. }
\date{Submitted May, 28 2019. Revised June, 22 2020. arXiv: 1905.11823}
\keywords{Free energy barrier, large deviations, McKean--Vlasov equation, mountain pass theorem, optimal transport, space of probability measures.}

\begin{document}
\begin{abstract}
We show that the empirical process associated with a system of weakly interacting diffusion processes exhibits a form of noise-induced metastability. The result is based on an analysis of the associated McKean--Vlasov free energy, which, for suitable attractive interaction potentials, has at least two distinct global minimisers at the critical parameter value $\beta=\beta_c$. On the torus, one of these states is the spatially homogeneous constant state, and the other is a clustered state.  We show that a third critical point exists at this value. As a result, we obtain that the probability of transition of the empirical process from the constant state scales like $\exp(-N \Delta)$, with $\Delta$ the energy gap at $\beta=\beta_c$. The proof is based on a version of the mountain pass theorem for lower semicontinuous and $\lambda$-geodesically convex functionals on the space of probability measures $\cP_2(M)$ equipped with the $2$-Wasserstein metric, where $M$ is a complete, connected, and smooth Riemannian manifold. 
\end{abstract}
\maketitle

\section{Introduction}
In recent years, a lot of progress has been made in understanding the convergence of interacting particle systems to their hydrodynamic or mean-field limits at the level of the convergence of gradient flows (cf.~\cite{ADPZ11,ADPZ13, DPZ13,Fat16,EFLS16,FS16,KJZ18}). These limits are described by dissipative evolution equations which are driven by some macroscopic free energy with respect to some metric. This gradient flow structure allows for a characterisation of the stationary states of the system in terms of critical points and minimisers of the free energy. Hence, the free energy landscape and the underlying metric encode some of the system's dynamical properties.  In many applications, the free energy is usually a lower semicontinuous (l.s.c) function  with the space of probability measures $\cP_2(M)$ as its domain. Here $\mu \in \cP_2(M)$ represents the distribution of particle positions on some base manifold $M$. The appropriate metric for the gradient flow is usually the $2$-Wasserstein metric and its variants. For example, in~\cite{BB18}, the authors derive a local mean-field model as the gradient flow of the macroscopic free energy with respect to a modified Wasserstein metric. 

For macroscopic models originating from interacting particle systems, the free energy can exhibit multiple local minima corresponding to distinguished stationary states of the macroscopic system. In this case, one may want to understand typical transition times and transition states between two such distinct states in the presence of noise. A typical example of this is a classical particle moving in $\R^d$ along the gradient of some potential $V \in C^2(\R^d;\R)$, i.e. 
\begin{equation}\label{eq:ODE}
\dot{x}(t)= -\nabla V (x) \, , 
\end{equation}
with $x(0)= x_0 \in \R^d$. Let us assume that $V$ has exactly two distinct global minima $x_1,x_2 \in \R^d$, which are also the stationary points of~\eqref{eq:ODE}. If one considers these to be the states of interest, then a relevant question is how does the particle transition from one to the other under the influence of noise. To understand this, one considers the stochastic differential equation (SDE)
\begin{align}\label{eq:SDE}
dX_t = - \nabla V(X_t) \dx{t} + \sqrt{2\beta^{-1}} \dx{B_t} \, , 
\end{align}
where $B_t$ is a $\R^d$-valued Wiener process and $\beta>0$ is a parameter representing the strength of the noise in the system. In the setting of the above SDE, the question can be reframed as follows: given $X_0=x_1$, what is the probability that in some finite time $T>0$, we have that $X_T=x_2$. This question is answered, at least for $\beta \gg 1$, by the Freidlin--Wentzell theorem. In particular, it tells us that the family of processes $\set{X_t^\beta} \in C([0,T];\R)$ with $X_0=x_1$ satisfy a large deviations principle with good rate function $S: C([0,T];\R^d) \to \R \cup\set{+\infty}$ given by
\begin{align}
S(f):= 
\frac{1}{4} \int_0^T \abs[\big]{\dot{f}(t)+ \nabla V(f(t))}^2 \dx{t} \, ,
\end{align} 
whenever the above integral is finite and $+\infty$ otherwise. As a consequence of the above result, we have that, for any closed and measurable $\Gamma \subset C([0,T];\R^d)$
\begin{align}
\limsup_{\beta \to +\infty}\beta^{-1} \log\bP\bra{X_t^\beta \in \Gamma} \leq - \inf_{f \in \Gamma} S(f) \, .
\end{align} 
If we pick $\Gamma= \set{f \in C([0,T];\R^d): f(0)=x_1, f(T)=x_2}$, we obtain an upper bound on the probability that the process reaches $x_2$ given that it starts at $x_1$. Setting $T^*= \arg \max_{t \in [0,T]} (V(f(t))-V(f(0)))$, we can obtain the following lower bound  for~$f \in \Gamma$,
\begin{align}
S(f)&= \frac{1}{4} \int_0^{T} \abs{\dot{f}(t)+ \nabla V(f(t))}^2 \dx{t}  \\
&= \frac{1}{4} \int_0^{T^*} \abs{\dot{f}(t)- \nabla V(f(t))}^2 \dx{t}  + \int_0^{T^*} \dot{f}(t)\cdot \nabla V(f(t)) \dx{t} \\
&\quad + \frac{1}{4} \int_{T^*}^T \abs{\dot{f}(t)+ \nabla V(f(t))}^2 \dx{t} \\
&\geq V(f(T^*))-V(f(0))  \geq \inf_{f \in \Gamma} \bra*{V(f(T^*))-V(f(0))} =:c- V(f(0))  \, .
\end{align}
It turns out that $c>0$ is in fact a critical value of $V$, i.e.  there exists $x_3 \in \R^d$ such that $V(x_3)=c$ and $\nabla V(x_3)=0$. The reader will recognise this as the finite-dimensional version of the well-known mountain pass theorem. Setting $\Delta:= V(x_3)-V(x_1)$, we see that for $\beta$ sufficiently large
\begin{align}
\bP(X_t^\beta \in \Gamma) \lesssim \exp\bra[\big]{-\beta \, \Delta} \, \label{eq:ldpi}. 
\end{align}
Thus, the probability of the process reaching the new phase/state in time $T>0$ goes exponentially with $\beta$ with the rate given by the difference between the energies of the saddle point and the initial phase.  Thus, we can see that the process finds the path of least resistance to reach the new phase in agreement with the fundamental tenet of large deviations theory that \emph{``an unlikely event will happen in the most likely of the possible unlikely ways.''} These transitions correspond to the phenomenon of noise-induced metastability, i.e. the process is stable around $x_1$ for $\beta \gg 1$, but there is an exponentially small probability of it transitioning to $x_2$. 

The purpose of this paper is to obtain results in a similar flavour but in an infinite-dimensional setting. Specifically, we are interested in understanding how related phenomena, i.e. noise-induced transitions, occur in systems governed by the Wasserstein gradient flow of some free energy $I$, especially those that arise as mean-field limits of interacting particle systems. We consider the following system of $N$ interacting SDEs on $\T$ (the $d$-dimensional torus of side length $L>0$)
\begin{align}
dX_t^i &= -\frac{1}{N}\sum_{j=1}^{N} \nabla W(X_t^i-X_t^j) \dx{t} + \sqrt{2 \beta^{-1}} \dx{B_t^i} \,  \\
\textrm{Law}(\bar{X}^N_0)&=\prod_{i=1}^N \nu(x_i) \quad \bar{X}_t^N= \bra*{X_t^1,\dots, X_t^N}
\end{align}
where $\beta>0$ is a parameter, $W \in C^2(\T)$ is an interaction potential which is even along every coordinate, and $B_t^i$ are $\T$-valued independent
Wiener processes. 

Let $\mu^{(N)}(t):=N^{-1} \sum_{i=1}^N \delta_{X_i^t}$, then it is well known (cf.~\cite{sznitman1991topics}) that $\mu^{(N)}(t)$ as a measure-valued random variable converges in law to $\mu=\mu(x,t)$ for each $t>0$, where $\mu$ is a weak solution of the following PDE
\begin{align}\label{mvintro}
\partial_t \mu = \nabla \cdot \bra[\big]{\mu \nabla(\beta^{-1} \log \mu + W \star \mu) } \qquad\text{with}\qquad \mu(x,0)=\nu(x) \, .
\end{align}
The above PDE is commonly referred to as the McKean--Vlasov equation and can be rewritten as $W_2$-gradient flow 
\begin{align}\label{eq:GF}
\partial_t \mu = \nabla \cdot\bra*{\mu \nabla \frac{\delta I}{\delta \mu} } \, ,
\end{align}
where $I : \cP(\T) \to \R \cup \set{+\infty}$ is the associated free energy. Its domain is the space of  absolutely continuous measures and for those it is given by
 \begin{align}\label{eq:I}
I(\mu) = \beta^{-1} \int \! \log \bra*{\frac{\dx{\mu}}{\dx{x}}} \, \dx{\mu} + \frac{1}{2} \iint \! W(x-y) \,  \dx{\mu}(y) \, \dx{\mu}(x) \, ,
 \end{align}
where $\frac{\dx{\mu}}{\dx{x}}$ denotes the density of $\mu$ with respect to the Lebesgue measure $\dx{x}$ on $\T$.  
The first term in~\eqref{eq:I} is referred to as the entropy and the second as the interaction energy. The function $I$ is referred to as the free energy of the system. The balance between entropy and interaction energy in terms of $\beta$ determines what the minimisers of $I$ look like. For $\beta$ smaller than some critical value $\beta_c$, the normalised Lebesgue measure, $\mu^L(\dx{x})=L^{-d}\dx{x}$ is the unique minimiser of the free energy. Above the value, $\beta_c$, a new minimiser of the free energy, which is not $ \mu^L $, emerges. The change in the structure of the set of minimisers of $I$ is called a phase transition and is observed in many models from the physical sciences~\cite{LP66, Sin82,Dawson1983,Shiino1987,GomesPavliotis2018, FV18}.

This operator $\nabla \cdot (\mu \nabla \frac{\delta}{\delta \mu}(\cdot))$ can be formally thought of as a gradient in the space of probability measures on $\T$ equipped with the $2$-Wasserstein mass transportation distance, which is defined as follows
 \[
W_2^2(\mu,\nu):= \inf_{\pi \in \Pi(\mu, \nu)} \int_{\T \times \T }d_{\T}(x,y)^2 \dx{\pi}(x,y) \, ,
 \]
where $\Pi(\mu,\nu)$ is the set of all couplings between $\mu$ and $\nu$ and $d_{\T}(\cdot,\cdot)$ is the distance on~$\T$. We note that $\cP(\T)$ equipped with $W_2$ is a complete, separable metric space. For $\mu,\nu$ absolutely continuous with respect to $\dx{x}$, the definition of the metric can be recast into the form discussed in~\thref{thm:monge}. This notion of a gradient flow can be made rigorous and is an extremely active field of research. However, the present work relies on quite classical results (cf.~\cite{cordero01,mccann01,mccann1997convexity,ambrosio2008gradient}). Indeed, the solutions of the McKean--Vlasov PDE are curves of maximal slope of the McKean--Vlasov energy $I$ with respect to $W_2$ (see~\cite{DS10}). Comparing this with the toy model discussed further up in the introduction, we see that the PDE has a gradient structure in $W_2$ and so the functional $I$ will play a similar role to the potential $V$ in~\eqref{eq:ODE}. The distinct phases/states are then characterised by the global minima of the functional~$I$ over $\cP(\T)$. The role of the SDE in~\eqref{eq:SDE} is then played by the empirical process $\mu^{(N)}$ and that of the parameter $\beta$ is played by $N$. In this context, we also refer to some recent progress in the understanding of singular SPDEs related to the fluctuations of the empirical process around its mean-field limit~\cite{FehrmanGess2019,Renesse2018,Zimmer2018,Zimmer2018b}.

Understanding such noise-induced transitions requires two ingredients: a version of the mountain pass theorem in the space of probability measures $\cP_2(M)$ equipped with the Wasserstein metric and an appropriate large deviations principle for the underlying particle system. We focus on the first ingredient noting that the second ingredient is usually application-specific. Our main result in this direction is as follows.
\begin{thm}\thlabel{thm:mpwi}
Assume $M$ is a complete, connected, and smooth Riemannian manifold. Let $I : \cP_2(M) \to \R \cup \set{+\infty}$ be a proper, l.s.c, and $\lambda$-geodesically convex functional. Suppose $\mu,\nu \in \cP_2(M) \cap D(I)$, $\Gamma$ is the set of all continuous curves $\gamma:[0,1] \to \cP_2(M)$ (where $\cP_2(M)$ is equipped with the 2-Wasserstein metric, $W_2$) with $\gamma(0)=\mu$ and $\gamma(1)=\nu$, and the function $\Upsilon: \Gamma \to \R$ is defined by:
\[
\Upsilon(\gamma)=\sup_{t \in[0,1]} I(\gamma(t)) \,.
\]
Let $c=\inf_{\gamma \in \Gamma} \Upsilon(\gamma)$ and $c_1=\max\set{I(\mu),I(\nu)}$. If $c>c_1$ and $I$ satisfies $\mathrm{(\hyperlink{MPS}{MPS})}$ (see Assumption~\eqref{ass:MPS}), then 
$c$ is a critical value of $I$, that is there exists a $\eta \in \cP_2(M)$ with $I(\eta)=c$ such that $\abs{\partial I}(\eta)= \abs{d I}(\eta)=0$ (see~\thref{def:rp} and~\thref{def:ms}). 
\end{thm}
The proof utilises the notion of the weak metric slope $\abs{dI}$ first introduced in~\cite{katriel94}. The main advantage over previous results in this direction is that we can apply the result to l.s.c functionals on $\cP_2(M)$ as long as they are $\lambda$-convex by working with the extension of the function to its epigraph based on ideas discussed by Degiovanni and Marzocchi~\cite{degiovanni94} originating from work in~\cite{DGMT80}. In fact for $\lambda$-convex functionals one can identify the usual (strong) metric slope $\abs{\partial I}$ and $\abs{dI}$. We focus on the case in which the metric is $W_2$ although the results generalise for $W_p$ or other variants of the metric.

Our first result shows that the abstract mountain pass \thref{thm:mpwi} can be applied to the McKean--Vlasov free energy $I$, as defined in~\eqref{eq:I}, after verifying the necessary regularity assumptions.
\begin{thm}\thlabel{intro:mpf}
Assume $W$ and $\beta$ are such that there exist two measures $\mu,\nu \in \cP(\T)$ such that $\mu$ is a strict local minimum of the McKean--Vlasov free energy $I$ (cf.~\eqref{eq:I}) and $I(\nu)\leq I(\mu)$. Then, there exists $\mu^* \in \cP(\T)$, distinct from $\mu$ and $\nu$, such that $\abs{\partial I}(\mu^*)=\abs{d I}(\mu^*)=0$. Additionally, $I(\mu^*)=c$, where $c$ is given by
\begin{align}
c= \inf\limits_{\gamma \in \Gamma} \sup_{t \in[0,1]} I(\gamma(t)) \, ,
\end{align}
where $\Gamma= \set*{C\bra[\big]{[0,1]; \cP(\T)}: \gamma(0)=\mu, \gamma(1)=\nu}$.  
\end{thm}
The proof of the above result can be found in Section~\ref{s:mva}. Furthermore, in Section~\ref{s:mva}, by relying on results from~\cite{CGPS18}, we will show that we can establish explicit conditions on the interaction potential~$W$  such that two distinct global minimisers $\set[\big]{\mu^L, \bar\mu}$ of the free energy~$I$~\eqref{eq:I} exist. This happens at a so-called \emph{discontinuous transition point} $\beta_c>0$. This provides us with a scenario in which we can apply~\thref{intro:mpf}.  Hereby, $\mu^L:=L^{-d}\dx{x}$ is the uniform state and $\bar \mu$ is a clustered state.  The existence of the associated saddle point can be found in~\thref{cor:mpf}.
\begin{rem}\thlabel{rem:rev2}
The abstract mountain pass theorem in~\thref{thm:mpwi} holds whenever one can find two measures, not necessarily critical points, in the domain of some $I:\cP_2(M) \to \R \cup \set{+\infty}$, such that the barrier value $c$ exceeds the maximum of their energies. We have chosen to apply the result in~\thref{intro:mpf} at a strict local minimum of the McKean--Vlasov free energy, $I$, because in this setting it is clear that the barrier value exceeds the value at the local minimum.
\end{rem} 
The fact, from~\thref{cor:mpf}, that the free energy functional $I$ has an energy barrier at $\beta=\beta_c$ allows us to study escape probabilities for the underlying particle system using results which were first proved by Dawson and G\"artner~\cite{DG1987}. We refer the reader to~\cite{ADPZ11,Reygner18,GPY13} for further discussions of the connections between large deviations theory and theory of gradient flows. 
\begin{thm}\thlabel{thm:ldi}
Assume $W$ and $\beta_c$ are such that there exist at least two distinct minimisers $\set{\mu^L, \bar \mu}$ of $I$. It follows then that the underlying empirical process $\mu^{(N)} \in \cC_T$ 
with initial i.i.d uniformly distributed particles satisfies
\begin{align}
\mathbb{P}(\mu^{N}(T)\in \overline{B}_\eps^{W_2}\bra{\bar{\mu}}, \mu^{(N)}(0)=\mu^{(N)}_0) \leq \exp\bra*{-N (\Delta- O(\eps^2))- o_T(1)}
\end{align}
for $N$ sufficiently large, where $\overline{B}_\eps^{W_2}\bra{\bar{\mu}}$ is the closed ball of size $\eps>0$ around $\bar{\mu}$ in $W_2$, $\Delta:= I(\mu^*)-I(\mu^L)$ and $\mu^*$ is the critical point defined in~\thref{cor:mpf}.
\end{thm}
The above result says that the probability of the empirical process reaching the clustered state, $\bar{\mu}$, from the uniform state, $\mu^L$, in time $T>0$ becomes exponentially small as the number of particles increases, as long as the system is at a discontinuous transition point.
In light of the recent results in~\cite{FehrmanGess2019}, we expect the above bound to hold for a stochastic version of the McKean-Vlasov PDE in the regime of vanishing noise, i.e.  Freidlin--Wentzell-type large deviations.

\begin{rem}\thlabel{rem:whyTd}
 We note that we have considered the McKean--Vlasov system on the torus in this paper which on the one hand has the advantage that the space $\cP(\T)$ equipped with the $2$-Wasserstein distance is compact, while on the other hand we can build on the characterisation of critical points and phase transitions from the work in~\cite{CGPS18}. Thus, we can extract a lot of information about the structure of stationary states in this setting. On the torus, the normalised Lebesgue measure $\mu^L=L^{-d} \dx{x}$ is always an invariant measure for the McKean--Vlasov dynamics~\eqref{mvintro} and it is the unique minimiser of the free energy before the critical temperature. Thus, linearisation arguments provide us with a lot of useful information that may not be readily available for the diffusions on $\R^d$. An extension to $\R^d$, with some suitable confinement potential, seems to be possible. The critical points and phase transitions are studied for specific choices of $W$, for example, in~\cite{tamura1984asymptotic} and~\cite{tugaut2013phase}. For the case of diffusions on $\R^d$ with a bi-stable confinement and Curie--Weiss-type quadratic interaction, large deviations, escape probabilities, and tunnelling results can be found in~\cite{DG1986,DG1987,DG1989} while a study of the basins of attraction of the different stationary states is the content of~\cite{Bashiri2020}.
\end{rem}

\begin{rem}\thlabel{rem:hclinic}
Even though the Dawson--G\"artner large deviations principle provides an exponential lower bound on the above probability, it is not clear how this can be compared to the energy barrier $\exp\bra[\big]{-N (\Delta- O(\eps^2))}$ for a general model. 
However, such a lower bound could be obtained, for example, in the following setting: for all $\eps>0$, there exist two points in $\mu^*_0, \bar\mu^*$ in a neighborhood of $\mu^*$ such that $\mu^*_0$ is connected to $\mu^L$ and $\bar\mu^*$ is connected to $\bar{\mu}$ through a heteroclinic orbit under the flow of the McKean--Vlasov PDE. However, it is unlikely that such heteroclinic connections exist at this level of generality in the choice of $W$. A first step in this direction would be the characterisation of $\mu^*$ for specific choices of $W$. In analogy to the situation in finite-dimensional Hamiltonian dynamical systems, it may be possible to use a version of weak KAM theory in the Wasserstein space of probability measures to construct such heteroclinic orbits. We refer the reader to~\cite{FengKatsoulakis2009,FengNguyen2012,AmbrosioFeng2014} for the construction of solutions to Hamiltonian--Jacobi PDEs in the space of probability measures. These ideas may help in obtaining a generalisation of weak KAM theory to dissipative evolution equations in the space of probability measures.
\end{rem}
\subsection*{Outline}
The paper is organised as follows: In Section \ref{s:cp}, we introduce the notion of the weak metric slope and metric critical points that we will use throughout the paper and a version of the mountain pass theorem due to Katriel that holds for continuous functions on metric spaces. In Section~\ref{s:ot}, we briefly recall some results due to McCann on optimal transport on Riemannian manifolds. In Section~\ref{s:mpw}, we compare the notion of the weak metric slope with the notion of (strong) metric slope used in the gradient flows community and show that, under the assumption of $\lambda$-convexity of $I$, the two are equivalent. We conclude the section by proving~\thref{thm:mpwi}. In the final Section~\ref{s:mva}, we discuss as a specific application of the result: the McKean--Vlasov model. We state and extend some results from~\cite{CGPS18} on the structure of the set of minimisers of $I$ and their phase transitions. We proceed by showing the existence of mountain pass at the point of discontinuous phase transition, thus proving~\thref{intro:mpf}. Finally, we introduce the precise form of the large deviations principle due to Dawson and Gärtner and complete the proof of~\thref{thm:ldi}. In particular, these results imply a kind of noise-induced metastability for the underlying particle system.

\section{Critical points in metric spaces}\label{s:cp}
We will assume throughout this section that $(\cX,d)$ is a complete metric space. We start with the definition of the weak metric slope for some real-valued continuous function defined on  $\cX$. The notion goes back to Ioffe and Schwartzman~\cite{IS96} who provided the definition in the Banach space setting.
\begin{defn}[$\delta$-regular points, weak metric slope and critical points~\cite{katriel94}] \thlabel{def:rp}
Let $x \in \cX$, and $I:\cX \to \R$ be a continuous function defined in a neighbourhood of $x$. Given $\delta>0$, $x$ is said to be a \emph{$\delta$-regular point} of $I$ if there is a neighbourhood $U$ of $x$, a constant $\alpha>0$, and a continuous mapping $\psi:  U \times [0,\alpha] \to \cX$ such that for all $(u,t) \in U \times [0,\alpha]$, it holds:
\begin{enumerate}
\item $d\bra*{\psi(u,t),u}\leq t$.
\item $I(u)-I(\psi(u,t))\geq \delta t$.
\end{enumerate}
If this is the case, $\psi$ is called a \emph{$\delta$-regularity mapping} for $I$ at $x$ and $x$ is called a \emph{regular point} of $I$.

The \emph{weak metric slope} of $I$ at $x$ is given by the extended real number
\[
    \abs{dI}(x)= \sup \set{\delta \in (0,\infty): I\textrm{ is } \delta\textrm{-regular at } x} \,.
\]
If $x$ is not $\delta$-regular for any $\delta>0$, then $x$ is called a \emph{critical point} of $I$ with $\abs*{dI}(x)=0$.
\end{defn}
\begin{assumption}[Weak metric Palais--Smale condition]\label{ass:MPS}
 A function $I:\cX \to \R$ is said to satisfy the \emph{weak metric Palais--Smale condition} (\hypertarget{MPS}{MPS}) if any \emph{Palais sequence}, that is $\{u_n\}_{n \in \N} \in \cX$ with $I(u_n )\to c \in\R$ and $|dI|(u_n) \to 0$, possesses a convergent subsequence. 
\end{assumption}
Given this notion, we have the following generalisation of the Ambrosetti--Rabinowitz mountain pass theorem due to Katriel~\cite{katriel94}.
\begin{thm}~\thlabel{thm:mpc}
Let $\cX$ be a path-connected metric space and $I: \cX \to \R$ be continuous. For $u_0,u_1 \in X$ let $\Gamma$ be the set of all continuous curves $\gamma:[0,1] \to \cX$ with $\gamma(0)=u_0$ and $\gamma(1)=u_1$, and the function $\Upsilon: \Gamma
\to \R$ is given by
\[
\Upsilon(\gamma)=\sup_{t \in[0,1]} I(\gamma(t)) \,.
\]
Let $c=\inf_{\gamma \in \Gamma} \Upsilon(\gamma)$ and $c_1=\max\set{I(u_0),I(u_1)}$. If $c>c_1$ and $I$ satisfies $\mathrm{(\hyperlink{MPS}{MPS})}$, then $c$ is a critical value of $I$.
\end{thm}
For the application of the mountain pass theorem to the energies encountered in gradient flows, we need a working definition of the weak slope if $I$ is only lower semicontinuous. We also need to deal with the fact that~\thref{thm:mpc} holds only for continuous functions. 
Consider the example $I : \R \to \R$ with $I(x)=x+1$ for $x <0$ and $I(x)=x$ for $x \geq 0$. Then $I$ is l.s.c, and it is easy to verify that $I$ has a critical point at $x=0$ in the sense of \thref{def:rp}. However, this seems to be in some sense pathological for identifying mountain pass points as $I$ does not attain the value of the barrier at $x=0$, i.e. $I(0)=0$. We like to use a theory of critical points for l.s.c. functionals which handles such critical points in our generalisation of the mountain pass theorem. One possible resolution for this example would be to view the function $I$ as a multivalued map at $x=0$ with values $[0,1]$ from which the energy barrier ist identified as the maximum of those values. Another resolution for this issue was suggested by Degiovanni and Marzocchi~\cite{degiovanni94}, who, using notions developed in~\cite{DGMT80}, proposed a generalisation based on the so-called \emph{epigraph extension}, $\cG_I$, a continuous function associated to $I$ and defined on its epigraph (cf.~\thref{def:epigraph}). This idea also helps us overcome the difficulty that~\thref{thm:mpc} holds only for continuous functions.
\begin{defn}[Extension to the epigraph]\thlabel{def:epigraph}
Let $I :\cX \to \R \cup \{+ \infty\}$ be a proper l.s.c functional and denote by $\epi(I)= \set{(u,\xi) \in \cX \times \R: I(u) \leq \xi }$ its epigraph, which is equipped with the \emph{graph metric} $d_{\epi} \bra[\big]{(u,\xi),(v,\zeta)}= \sqrt{d(u,v)^2 + |\xi -\zeta|^2}$. The \emph{epigraph extension} of $I$ is the functional $\cG_I: \epi(I) \to \R \cup \{+\infty\}$ defined by 
\[
  \cG_I(u,\xi)=\xi, \qquad (u, \xi) \in \epi(I) \, .
\]
\end{defn}
It is now straightforward to check that $\cG_I$ is a continuous function with respect to $d_{\epi}$ and that $\abs*{d \cG_I}(u,\xi) \leq 1$ for all $(u,\xi) \in \epi(I)$.
Let us also point out that the epigraph of a lower semicontinuous function is closed~\cite[Proposition 2.5]{BarbuPrecupanu2012}.
It turns out, that the notion of the weak slope for l.s.c functions on $\cX$ based on the epigraph extension is suitable for applications to mountain pass theorems in metric spaces. 
\begin{defn}~\thlabel{def:wms}
Let $I : \cX \to \R \cup \{+\infty\}$ be a proper l.s.c function. Define its domain as
\[
D(I) := \set*{x \in \cX : I(x) < + \infty} \,.
\]
Then the \emph{weak metric slope} at $x\in D(I)$ is defined by
\begin{align}
\abs*{d I(x)}= 
\begin{cases}
\frac{\abs*{d \cG_I(x,I(x))}}{\sqrt{1- \abs*{d \cG_I(x,I(x))}^2}} &\text{ if } \abs{d \cG_I(x,I(x))} <1 \\
+ \infty  &\text{ if } \abs{d \cG_I(x,I(x))} =1 \, . 
\end{cases}
\end{align}
Again, $x \in D(I)$ is called \emph{critical point} of $I$ if $(x,I(x)) \in \epi(I)$ is a critical point of $\abs*{d \cG_I}\bra[\big]{u,I(u)}$.
\end{defn} 
In the case when $I$ is continuous the above definition is equivalent to~\thref{def:rp}. Indeed, it holds by \cite[Proposition 2.3]{degiovanni94}, that in this case 
\[
  \abs*{d\cG_I(x,I(x))} = \begin{cases}
  \frac{|dI(x)|}{\sqrt{1+\abs*{d I(x)}^2}} &\text{ if } \abs*{dI(x)} < \infty \\
  1 &\text{ if } \abs{dI(x)} = \infty
                         \end{cases} \quad\text{and}\quad \abs{d\cG_I(x,\xi)}=1 \text{ if } I(x) < \xi \, . 
\]
Hence, the~\thref{def:wms} is a generalisation of the weak metric slope from~\thref{def:rp} to lower semicontinuous functionals. However, this definition is, in general, hard to verify. For this reason, we state without a proof a result from~\cite{degiovanni94} that provides a lower bound on $\abs{dI}$.
\begin{prop}[{\cite[Proposition 2.5]{degiovanni94}}]~\thlabel{prop:prop2.5}
Let $I : \cX \to \R \cup \set{+\infty}$ be a proper, l.s.c functional and for $b\in \R$ let $D(I)_b = \set{ x\in D(I) : I\leq b}$. If for some $x \in D(I)$ there exist constants $\delta>0, b> I(x), \alpha >0$, a neighbourhood $U$ of $x$, and a mapping $\Psi: (U \cap D(I)_b )\times [0, \alpha]\to \cX$ such that for all $(u,t) \in U \cap D(I)_b \times [0, \alpha] $ it holds that
\[
d(\Psi(u,t),u) \leq t \quad\text{and}\quad I(u)-I(\Psi(u,t)) \geq \delta \, t \, .
\]
Then, $\abs{dI}(x) \geq \delta$.
\end{prop}
We return to the previous example $I: \R \to \R$ with $I(x)=x+1$ for $x<0$ and $I(x)=x$ for $x\geq 0$. 
In regard of~\thref{def:rp}, 
we pick $U$ to be the ball of size $\delta$ around $(0,0)$ in $\epi(f)$.
Choosing the map $\Phi\bra[\big]{(x,\xi),t}=\bra[\big]{x+t/\sqrt{2},\xi+t/\sqrt{2}}$, then we have $\cG_f((x,\xi))=\xi$

we have that $\abs*{d \cG_f}((0,0))=1$
and thus $\abs*{d f}(0)= +\infty$. Thus, the new definition captures the fact that $f$ has a jump at $x=0$ and correctly does not classify it as a critical point. 

Although we can apply the mountain pass~\thref{thm:mpc} to the function $\cG_I$, we do not know if the critical point we obtain is of the form $\bra[\big]{x,I(x)}$, i.e. we have no information about how $\abs{d\cG_I}$ behaves away for $(x, \xi) \in \epi(I)$ such that $\xi>I(x)$.  Degiovanni and Marzocchi~\cite{degiovanni94} provide some intuition in the case in which $I$ is a functional defined on a Banach space and consists of convex l.s.c part plus a $C^1$ perturbation. The critical point, in this case, is defined relative to the metric generated by the norm. Another more abstract approach could be to establish a mountain pass theorem for multivalued maps, as indicated in the discussion surrounding~\thref{def:epigraph}.

In the study of gradient flows, this problem can be treated differently. In Section~\ref{s:mpw}, we show how the notion of $\lambda$-convexity ensures that critical points in the sense of \thref{def:wms} are actually of the form $\bra[\big]{x,I(x)}$ (cf.~\thref{lem:slopeaway}). Furthermore, the notion of weak metric slope for l.s.c functions introduced in~\thref{def:wms} ensures that the critical points obtained in~\thref{thm:mpwi} attain the barrier value. Before discussing this in further detail, we first cover some preliminaries on optimal transport on Riemannian manifolds.

\section{Optimal transport on manifolds} \label{s:ot}
Let $M$ be a complete, connected, and smooth Riemannian manifold equipped with a metric given in local coordinates by $g_{ij}$. We denote the geodesic distance between $x,y \in M$ by $d_M(x,y)$ and the Riemannian volume element by $\dx{\vol}(x)=\sqrt{\det g
_{ij}(x)}\dx{x}$ in local coordinates. For $x \in M$, we denote the inner product on the tangent space $T_x M$ by $\skp{\cdot,\cdot}$.
Let $c(x,y):=d_M(x,y)^2/2$ denote the cost function. We denote by $\cP_2(M)$ the space of Borel probability measures on $M$ with finite second moment. Given $\mu,\nu \in \cP_2(M)$, the $2$-Wasserstein distance, $W_2(\cdot,\cdot)$, between them is defined as
\begin{align}
W_2(\mu,\nu):= \inf_{\pi \in \Pi(\mu,\nu)} \int_{M \times M}d_M^2(x,y) \dx{\pi}(x,y) \, ,
\end{align}
where $\Pi(\mu,\nu)$ is the set of all couplings between $\mu$ and $\nu$. We first discuss the existence of minimising geodesics in $\cP_2(M)$ in the absence of any regularity assumptions on the initial and final measures. Given a curve $\mu \in C([0,1];\cP_2(M)) $, we define its length to be
\begin{align}
L(\mu):= \sup_{N \in \N}\, \sup_{0= t_0< \dots < t_N=1} \sum_{i=0}^{N-1}W_2(\mu\bra{t_i},\mu\bra{t_{i+1}}) \, .
\end{align}
We now state the following result from~\cite[Corollary 7.2.2]{villani2008optimal}:
\begin{prop}\thlabel{prop:geodesic}
The space $\cP_2(M)$ is a geodesic metric space, i.e. for any two $\mu_0, \mu_1 \in \cP_2(M)$ there exists a minimising geodesic between them. That is to say, there exists a curve $\mu \in C([0,1]; \cP_2(M))$ between $\mu_0$ and $\mu_1$, such that
\begin{align}
W_2(\mu_0,\mu_1)= \min\set[\Big]{ L(\overline\mu): \overline\mu \in C([0,1];\cP_2(M)), \overline\mu(0)=\mu_0, \overline\mu(1)=\mu_1}= L(\mu) \, .
\end{align}
Furthermore, the curve $\mu$ can be reparametrised to have unit speed, that is
\begin{align}
W_2(\mu(t),\mu(s))= \abs{t-s}W_2(\mu_0,\mu_1) \, ,
\end{align}
for all $s,t \in [0,1]$. Such a reparametrised curve $\mu \in C([0,1];\cP_2(M))$ is called a \emph{unit speed minimising geodesic} between $\mu_0$ and $\mu_1$.
\end{prop}
\begin{rem} \thlabel{rem:lspace1}
We will use the above proposition extensively to prove the mountain pass theorem in $\cP_2(M)$, i.e. ~\thref{thm:mpwi}. Note, however, that the result of~\thref{prop:geodesic} holds even if $M$ is replaced by some $\cM$, where $\cM$ is a complete, separable, and locally compact, length space (as is any complete, connected, Riemannian manifold, by the Hopf--Rinow theorem). However, we will refrain from working at this level of generality and will instead remind the reader when any stated results can be generalised to $\cM$. 
\end{rem}
We now introduce the following definition following~\cite{cordero01}.
\begin{defn}
  Let $A,B$ be compact subsets of $M$. The set $\cI^c(A,B)$ of \emph{$c$-concave functions} is the set of functions $\phi:A \to \R \cup \set{-\infty}$ not identically $-\infty$, for which there exists a function $\psi: B \to \R \cup\set{-\infty}$ such that
  \[
    \phi(x)= \inf_{y \in B}\bra[\big]{ c(x,y) -\psi(y)}, \quad \forall x \in A \,.
  \]
  The function $\phi$ is called the \emph{$c$-transform} of $\psi$ and abbreviate it as $\phi=\psi^c$.
\end{defn}
We have the following main result on the well-posedness of the Monge problem from~\cite{mccann01}.
\begin{thm}~\thlabel{thm:monge}
Let $M$ be a complete Riemannian manifold. Fix two Borel probability measures $\mu  \ll \vol$ and $\nu$ on $M$ and two compact
subsets $A,B \subset M$ containing the supports of $\mu$ and  $\nu$, respectively. Then there exists a $\phi \in \cI^c(A,B)$
such that the map
\[
  F(x) := \exp_x\bra*{-\nabla \phi (x)} \qquad\text{ is a pushforward of $\mu$ to $\nu$.}
\]
Furthermore, $F$ is the unique minimiser of the quadratic cost $\int_M c(x, G(x)) \dx{\mu}(x)$ among all Borel maps $G:M \to M$ pushing $\mu$ forward to $\nu$ apart from variations on sets of $\mu$-measure zero. It follows then that the $W_2$ transportation distance between $\mu$ and $\nu$ takes the following form
\[
W_2^2(\mu,\nu)= \int_M d_M(x, F(x))^2 \dx{\mu}(x) \,.
\]
\end{thm}
The natural extension on McCann's notion of displacement interpolation~\cite{mccann1997convexity} to the manifold setting is given in the following definition.
\begin{defn}[Optimal interpolant]~\thlabel{def:opip}
Let $M$ be a complete Riemannian manifold. Fix two Borel probability measures $\mu  \ll \vol$ and $\nu$ on $M$ and two compact
subsets $A,B \subset M$ containing the supports of $\mu$ and  $\nu$, respectively. We define the \emph{optimal interpolant} to
be the map $t \mapsto \mu(t)$ for $t \in[0,1]$ such that $\mu(t) = (F_t)_{\#}\mu$ and $F_t= \exp_x\bra*{-t \nabla \phi(x)}$. Here
$\phi \in \cI^c(\bar{A},B)$ is the so-called Kantorovich potential between $\mu$ and $\nu$ from~\thref{thm:monge}.
\end{defn}
We are finally in a position to conclude this section with the following results from \cite{cordero01} about the properties of
the optimal interpolant.
\begin{lem} \thlabel{lem:oiprop}
Let $M$ be a complete Riemannian manifold. Fix two Borel probability measures $\mu  \ll \vol$ and $\nu$ on $M$ and two compact
subsets $A,B \subset M$ containing the supports of $\mu$ and  $\nu$, respectively. Then the following two results hold
\begin{tenumerate}
\item \emph{Optimality of the optimal interpolant.} The map $F_t$ defined in~\thref{def:opip} is the minimiser of the quadratic cost between $\mu(t)$ and $\mu$ among all maps pushing forward $\mu$ to $\mu(t)$ for all $t \in [0,1]$. \label{lem:oi1}
\item \emph{Absolute continuity of the interpolant.} If $\mu$ and $\nu$ are compactly supported absolutely continuous with respect to the Riemannian volume, then so is their optimal interpolant $t \to \mu(t)$ for all $t \in[0,1]$. 
\end{tenumerate}
\end{lem}
\section{A mountain pass theorem in \texorpdfstring{$\cP_2(M)$}{P(M)}} \label{s:mpw}
We now turn to the question of obtaining a notion of mountain passes for l.s.c functions. We fix our metric space
to be $\cX = \cP_2(M)$, where $M$ is now a complete connected smooth Riemannian manifold, and we equip it with the $d=W_2$
transportation distance which makes it a complete, separable metric space~\cite{villani2008optimal}. The functionals under consideration satisfy a geodesic $\lambda$-convexity assumption introduced in the following definition.

\begin{defn}[Geodesic $\lambda$-convexity]\thlabel{def:lconvex}
A proper l.s.c function $I : \cP_2(M) \to \R \cup \set{+\infty}$ is said to be \emph{$\lambda$-geodesically convex} for some 
$\lambda \in \R$, if for any $\mu_0,\mu_1 \in \cP_2(M) \cap D(I)$ it holds that $I(\mu(t))$, where $\mu \in C([0,1];\cP_2(M))$  is any unit speed minimising geodesic between $\mu_0$ and~$\mu_1$ (cf.~\thref{prop:geodesic}), satisfies
\[
I(\mu(t)) \leq (1-t) I(\mu_0) + t I(\mu_1) -\frac{\lambda}{2}t (1-t) W_2^2(\mu_0,\mu_1) \, \quad \forall t \in[0,1].
\]
\end{defn}
The following lemma, whose proof is similar in spirit to \cite[Theorem 3.13]{degiovanni94}, shows that the weak metric slope of $\cG_I$ is non-zero for geodesically $\lambda$-convex functionals for $(\mu,\xi) \in \epi(I)$ such that $\xi>I(\mu)$. In particular, any critical point of $\cG_I$, if present, satisfies $\xi=I(\mu)$.
\begin{lem}~\thlabel{lem:slopeaway}
Let $I : \cP_2(M) \to \R \cup\set{+\infty}$ be a proper, l.s.c, and $\lambda$-geodesically convex function.
Then, it holds for all $\mu\in \cP_2(M)$ and $\xi\in \R$ that
\begin{align}
\abs{d\cG_I}(\mu, \xi) =1 \qquad \textrm{if }\xi > I(\mu) \, .
\end{align}
In particular, any critical point $(\mu,\xi)$ of $\cG_I$ satisfies $\xi=I(\mu)$.
\end{lem}
\begin{proof}
Let $(\mu_1,\xi) \in \epi(I)$ be such that $\xi = I(\mu_1) + 2\eps$ for some $\eps>0$. We define for any $\delta>0$ the map $\Psi:B_\delta^{d_{\epi}}\bra{\mu_1,\xi} \times [0,\eps] \to \epi(I)$ as follows
\begin{align}\label{e:Psi:p0}
\Psi\bra[\big]{(\mu_0,\alpha),t} =  \bra*{\mu\bra*{\frac{t}{\Lambda}}, \alpha - \frac{t}{\Lambda}\bra*{\alpha -\frac{|\lambda|}{2} W_2^2(\mu_0,\mu_1)- I(\mu_1)}} \, ,
\end{align}
where 
\[
\Lambda=\sqrt{W_2^2(\mu_0,\mu_1) + \abs*{\bra*{\alpha -\frac{|\lambda|}{2} W_2^2(\mu_0,\mu_1)- I(\mu_1)}}^2 } 
\]
and $\mu(\cdot)$ is a unit speed minimising geodesic between $\mu_0$ and $\mu_1$. We need to first verify that $t/\Lambda \in [0,1]$. Since $\xi \geq I(\mu_1)+2\eps$, we find $\delta_0=\delta_0(\eps)$ such that for all $\delta \in (0,\delta_0)$ it holds
\begin{equation}\label{e:eps:p0}
 \eps \leq \xi - I(\mu_1) - \frac{|\lambda|}{2 }\delta^2 - 2 \delta  \, .
\end{equation}
The above estimate yields $\Lambda\geq \eps$ implying $\frac{t}{\Lambda} \in [0, 1]$ and so $\mu\bra{\frac{t}{\Lambda}}$ is well-defined, provided that $\delta\in (0,\delta_0)$. 
We also have, from~\thref{prop:geodesic}, that
\begin{align}
d_{\epi}\bra[\Big]{\Psi((\mu_0,\alpha),t),(\mu_0,\alpha)}=t \, .
\end{align}
Thus, the map $\Psi$ satisfies condition (1) of~\thref{def:rp}. We still have to check that $\Psi\bra[\big]{(\mu_0,\alpha),t} \in \epi(I)$. Indeed we have from the definition of $\lambda$-geodesic convexity
\begin{align}
I\bra[\big]{\mu\bra{\frac{t}{\Lambda}}} &\leq I(\mu_0) + \frac{t}{\Lambda}\bra[\big]{I(\mu_1)-I(\mu_0)} - \frac{\lambda}{2}\frac{t}{\Lambda}\bra*{1-\frac{t}{\Lambda}} W_2^2(\mu_0,\mu_1) \\
& \leq \alpha - \frac{t}{\Lambda}\bra*{\alpha -  I(\mu_1)} + \frac{|\lambda|}{2}\frac{t}{\Lambda} W_2^2(\mu_0,\mu_1) \\
& = \alpha - \frac{t}{\Lambda}\bra*{\alpha - \frac{|\lambda|}{2} W_2^2(\mu_0,\mu_1)-  I(\mu_1)}    \, .
\end{align}
Finally, we can proceed from~\eqref{e:Psi:p0} to the following estimate
\begin{align}
\cG_I\bra*{\Psi((\mu_0,\alpha),t)} &= \alpha - \frac{t}{\Lambda}\bra*{\alpha - \frac{|\lambda|}{2} W_2^2(\mu_0,\mu_1)-  I(\mu_1)}  \\
& \leq  \cG_I((\mu_0,\alpha)) - t \frac{\xi  -  I(\mu_1) -\delta -  \delta^2 \frac{|\lambda|}{2}}{\sqrt{\delta^2 + \abs*{\xi - I(\mu_1) + \delta + \delta^2 \frac{|\lambda|}{2}}^2 }}  \,.
\end{align}
Thanks to~\eqref{e:eps:p0}, we can make $\delta$ arbitrarily small and obtain that $\abs{d \cG_I}(\mu_1, \xi) \geq 1$ from~\thref{def:rp}~$(2)$. Since
$\abs{d \cG_I}(\mu_1, \xi) \leq 1$ by~\thref{def:epigraph}, the result follows.
\end{proof}
Having showed that the weak metric slope of $\cG_I$ is a constant equal to one for all points $(\mu,\xi) \in \epi(I)$ such that $\xi>I(\mu)$, we investigate how the critical points of $I$ defined through the weak metric slope relate to other relevant notions. Specifically, we compare it to the notion of  critical point derived from the strong metric slope used in theory of gradient flows~\cite{ambrosio2008gradient}. This theory makes rigorous the notion of the Wasserstein gradient discussed in the introduction. We briefly introduce some terminology.  Let $I: \cP_2(M) \to \R \cup \set{+\infty}$ be a proper, l.s.c, and $\lambda$-geodesically convex.
\begin{defn}[Absolutely continuous curves]~\thlabel{def:ac}
A curve $\mu: [a,b] \subset \R \to \cP_2(M)$ is said to belong to $AC^p([a,b];\cP_2(M))$ for some $p \in[1,+\infty]$ if there
exists $m \in \Leb^p([a,b])$ such that
\begin{align}
W_2(\mu(s),\mu(t)) \leq \int_s^t m(r) \dx{r}, \quad a \leq s \leq t \leq b \, .
\label{eq:ac}
\end{align}
If $p=1$, then $\mu$ is said to be an \emph{absolutely continuous curve}.
\end{defn}
\begin{thm}[Metric derivative]
If $\mu:[a,b] \to \cP_2(M)$ is an absolutely continuous curve then the limit
\[
\abs{\mu'}(t)= \lim_{s \to t} \frac{W_2(\mu(s),\mu(t))}{\abs{t-s}} \, ,
\]
exists for a.e. $t$ and is called the \emph{metric derivative} of $\mu$. Additionally, $|\mu'| \in \Leb^1([a,b])$ and is admissible as an $m$
in~\eqref{eq:ac}. In fact it is the minimal admissible $m$, i.e. 
\[
|\mu'|(t) \leq m(t)
\]
for $t$ a.e. where $m$ satisfies~\eqref{eq:ac}.
\end{thm}
Now we introduce the notion of the (strong) metric slope.
\begin{defn}[Metric slope]~\thlabel{def:ms}
The \emph{metric slope} $\abs{\partial I}$ of $I$ at $\mu \in \cP_2(M)$ is defined as 
\begin{align}
\abs{\partial I}(\mu)=
\begin{cases}
\limsup\limits_{\nu \to \mu} \frac{(I(\mu)-I(\nu))_+}{W_2(\mu,\nu)} & \mu \in D(I)\\
+\infty & \textrm{otherwise}  \, .
\end{cases}
\end{align}
\end{defn}
The metric slope is defined with a positive part, since we are interested in (negative) gradient flows decreasing the energy functional $I$ (see~\cite[Chapter 10]{ambrosio2008gradient}).
Finally, we are in a position to define the notion of a curve of maximal slope.
\begin{defn}[Curves of maximal slope]
A curve $\mu \in AC^2([0,+\infty); \cP_2(M))$ is a \emph{curve of maximal slope} of the function $I$ if the following energy dissipation inequality is satisfied 
\begin{align}
\frac{1}{2}\int_s^t |\mu'|^2(r) \dx{r} + \frac{1}{2} \int_s^t |\partial \phi|^2 (\mu_r) \dx{r} \leq I(\mu(s)) -I(\mu(t))) \,,
\end{align}
for all $0\leq s \leq t <+\infty$. A curve $\mu$ is a \emph{stationary curve} of maximal slope if it is a curve of maximal slope and $\mu(t)=\mu(s)$ for all $s,t \in[0,+\infty)$.
\end{defn}
We have the following straightforward corollary.
\begin{cor}
A curve of maximal slope $\mu$ of a function $I$ is stationary if and only if $\abs{\partial I}(\mu)=0$.
\end{cor}
Using all these notions we can finally compare the weak metric slope with the metric slope.
\begin{lem}[Equivalence of the two notions of slope]~\thlabel{prop:equivslope}
Let $I : \cP_2(M) \to \R \cup \set{+\infty}$ be a proper, l.s.c, and $\lambda$-geodesically convex functional. Then for $\mu \in D(I)$ it holds that $|dI|(\mu)=|\partial I|(\mu)$.
 \end{lem}
 \begin{proof}
 We first show that $|dI|(\mu)\leq |\partial I|(\mu)$.  Let $\cG_I$ be the continuous extension to the epigraph and let $\Psi$ be a $\delta$-regularity mapping for the point $(\mu, I(\mu))$ with $\mu \in D(I)$, that is by~\thref{def:rp}
\begin{align}
\frac{I(\mu) - \Psi((\mu,I(\mu)),t)}{d_{\epi}((\mu,I(\mu)),\Psi((\mu,I(\mu)),t))} \geq \delta \,.
\end{align}
At the same time, we can choose $\Psi((\mu,I(\mu)),t)$ as the approximating sequence in~\thref{def:ms} of the strong metric slope and obtain 
\begin{align}
\abs{\partial \cG_I}(\mu,I(\mu)) \geq \delta \, .
\end{align}
Taking the supremum over all such $\delta$, we have the bound $\abs{\partial \cG_I}(\mu,I(\mu)) \geq \abs{d \cG_I}(\mu,I(\mu)) $. This yields only the comparison of the two different slopes of the epigraph extension~$\cG_I$. To obtain the comparison of the slopes of the functional $I$ itself, we first assume that $\mu \in D(I)$ is not a local minimum, and $\abs{\partial I}(\mu)< + \infty$. Then there exists a sequence $(\nu_n, I(\nu_n)) \in \epi(I)$ such that it converges to $(\mu,I(\mu))$ and such that $I(\nu_n) \leq I(\mu)$ for all $n$ sufficiently large. Using this as the approximating sequence in~\thref{def:ms}, we obtain 
\begin{align}
\abs{\partial \cG_I}(\mu,I(\mu))  &= \limsup_{(\nu_n, I(\nu_n)) \to (\mu,I(\mu))}\frac{(I(\mu)-I(\nu_n))_+}{\sqrt{\abs{I(\mu)-I(\nu_n)}^2 + W_2^2(\mu,\nu_n)}} \\
&=\limsup_{(\nu_n, I(\nu_n)) \to (\mu,I(\mu))}\frac{(I(\mu)-I(\nu_n))_+}{\sqrt{\bra{I(\mu)-I(\nu_n)}_+^2 + W_2^2(\mu,\nu_n)}} \\
&= \frac{\abs{\partial I}(\mu)}{\sqrt{1 + \abs{\partial I}(\mu)^2}}
\end{align}
When $\mu \notin D(I)$ or $\mu$ is local minimum both $\abs{\partial I}(\mu)$ and $\abs{\partial \cG_I}(\mu)$ are $+\infty$ and $0$ respectively. Using~\thref{def:wms} of the weak metric slope, we have $|dI|(\mu)\leq |\partial I|(\mu)$.

To prove the other inequality, we first assume that $|\partial I|(\mu) =: \eps_0 > 0$. Then, for any $\eps \in (0,\eps_0)$ exists $\delta=\delta(\eps)>0$ by~\thref{def:ms} such that there exists $\nu \in B_{\delta}(\mu)$ with
\[
I(\nu) < I(\mu) - \eps W_2(\mu,\nu) \,.
\]
Choose such a $\nu$  and set $\delta'= W_2(\mu,\nu) < \delta$. Since $I$ is l.s.c, we find for any $n\in \N, n\geq 2$ some $\alpha=\alpha(\delta',n)>0$ such that for all $\eta \in B_\alpha(\mu)$ it holds that
\[
I(\mu)-I(\eta) \leq \frac{\delta'}{n} \, .
\]
We define $\alpha'=  \alpha'(\delta',n) = \min \set{\alpha, \delta'/n}$ and define a map $\Psi: B_{\alpha'}(\mu) \times 
[0,\alpha'] \to \cP_2(M)$ as follows
\[
\Psi(\eta,t)=  \gamma_{\eta,\nu}\bra*{\frac{t}{W_2(\nu,\eta)}} \,,
\]
where $\gamma_{\eta,\nu}\bra{\cdot}$ is any unit speed minimising geodesic between $\eta$ and $\nu$ (cf. ~\thref{prop:geodesic}). Again, we have to check that $0\leq t/W_2(\nu,\eta)\leq 1$. We have from the definition of $\alpha'$ by the triangle inequality
\begin{align}
\frac{n-1}{n} \delta' \leq - W_2(\mu,\eta)+ W_2( \mu,\nu) \leq W_2(\nu,\eta) &\leq W_2(\mu,\eta)+ W_2( \mu,\nu) \leq \frac{n+1}{n} \delta' \,.
\end{align}
Thus, it follows that $0 \leq t/W_2(\eta,\nu) \leq 1$. Also, by construction holds $W_2(\eta,\Psi(\eta,t))=t$. Now, by the $\lambda$-geodesically convexity of $I$, we obtain the following estimate
\begin{align}
I(\Psi(\eta,t)) &\leq I(\eta) +\frac{t}{W_2(\eta,\nu)} \bra[\big]{I(\nu)-I(\eta)} + \frac{|\lambda|}{2}t\bra*{1-\frac{t}{W_2(\eta,\nu)}}W_2(\eta,\nu) \\
&\leq  I(\eta) +\frac{t}{W_2(\eta,\nu)} \bra[\big]{I(\nu)-I(\mu)} + \frac{t}{W_2(\eta,\nu)} \bra[\big]{I(\mu)-I(\eta)} + \frac{|\lambda|}{2} \, t\, \delta' \frac{n+1}{n}   \\
& < I(\eta) + t \bra*{ -\eps\bra*{ \frac{n}{n+1} } + \frac{1}{n-1} + \delta'|\lambda| } \, .
\end{align}
We can pick $\delta'>0$ to be as small and $n$ as large as we want and conclude $I(\Psi(\eta,t)) \leq I(\eta) - \eps t$. It follows from~\thref{prop:prop2.5} that $|dI|(\mu)\geq \eps$. Thus, since $\eps\in (0,\eps_0)$ is arbitrary, we have that $|dI|(\mu) \geq \eps_0 =|\partial I|(\mu)$ for all positive values. For the case in which $\abs{ d I}(\mu)=0$, assume that $\abs{\partial I}(\mu)= \eps_0>0$. But we have shown that $\abs{\partial I}(u)=\abs{\partial I}(\mu)= \eps_0>0$ which would be a contradiction. Thus, we have $\abs{\partial I}(\mu)=\abs{\partial I}(\mu)$. 
 \end{proof}
\begin{prop}\thlabel{prop:connected}
Let $I : \cP_2(M) \to \R \cup \set{+\infty}$ be a proper, l.s.c, and $\lambda$-geodesically convex functional. Then $\epi (I)$ 
is complete and path-connected.
 \end{prop}
 \begin{proof}
  Since $\cP_2(M) \times \R$ is complete, we have for any convergent sequence $(\mu_n,\xi_n) \in \epi(I)$ converging to some $(\mu,c) \in \cP_2(M) \times \R$, that $I(\mu) \leq \lim \inf I(\mu_n)\leq \lim \inf\xi_n=c$. Thus, $\epi(I) \subset \cP_2(M) \times \R$ is closed and thus complete.

Let $(\mu_0,\alpha), (\mu_1,\beta) \in \epi( I)$. Then $(\mu(\cdot),(1-t)\alpha + t\beta - \frac{\lambda}{2}t(1-t)W_2^2(\mu_0,\mu_1))$, where $\mu \in C([0,1];\cP_2(M))$ is any unit speed minimising geodesic between $\mu_0$ and $\mu_1$ (cf.~\thref{prop:geodesic}), is a continuous path  (with respect to $d_{\epi}$) between them which lies entirely in $\epi( I)$.  
 \end{proof}
We conclude this section with the proof of~\thref{thm:mpwi}.
 \begin{proof}[Proof of~\thref{thm:mpwi}]
 Denote by $\Gamma_{\epi}$ the set of all continuous curves $\gamma_{\epi}:[0,1] \to \epi(I)$ with $\gamma_{\epi}(0)=(\mu,I(\mu))$ and $\gamma_{\epi}(1)=(\nu,I(\nu))$. We can identify any $\gamma_{\epi} \in \Gamma_{\epi}$ with a $\tilde{\gamma} \in \Gamma$ by projecting onto the first factor, i.e.  $(t \mapsto (\mu(t),\xi(t))) \mapsto (t \mapsto \mu(t))$. Because of the definition of the epigraph and $\cG_I$, we have that
\begin{align}\label{eq:cubound}
\inf_{\gamma_{\epi} \in \Gamma_{\epi}} \max_{t \in[0,1]} \cG_I(\gamma_{\epi}(t)) &\geq \inf_{\gamma_{\epi} \in \Gamma_{\epi}} \sup_{t \in[0,1]} I(\tilde{\gamma}(t)) \geq \inf_{\gamma \in \Gamma}  \sup_{t \in[0,1]} I(\gamma(t)) =c \,.
\end{align}
Now, we prove the inequality holds the other way as well. Note that, for every $\eps>0$, there must exist some $\gamma \in \Gamma$ such that 
\begin{align}
c \leq \sup_{t \in [0,1]} I(\gamma(t)) \leq c+\frac{\eps}{2} \, .
\end{align}
Consider a partition $\mathfrak{P}_\delta=\set{t_i}_{i=0 ,\dots, N}$ of $[0,1]$ having mesh size $\delta>0$. Consider the family of curves $\set{\gamma^{\delta,i}}_{i=0,\dots,N-1} \subset C([0,1];\cP_2(M))$  associated to the partition $\mathfrak{P}_\delta$, where $\gamma^{\delta,i}(\tau)= \mu^i(\tau)$, $\tau \in [0,1]$ and $\mu^i$ is a unit speed minimising geodesic between $\gamma(t_i)$ and $\gamma(t_{i+1})$. We can now construct a family of curves $\set{(\gamma^{\delta,i},\xi^{\delta,i})}_{i=0,\dots,N-1}\subset C([0,1];\epi(I))$ such that 
\[
\xi^{\delta,i}(\tau):=(1-\tau)I(\gamma(t_i)) + \tau(I(\gamma(t_{i+1}))) + \frac{\abs{\lambda}}{2}\tau(1-\tau)W_2^2(\gamma(t_i),\gamma(t_{i+1}))  \, .
\]   
The fact that $I$ is $\lambda$-geodesically convex ensures that $(\gamma^{\delta,i}(\tau),\xi^{\delta,i}(\tau)) \in \epi(I)$ for all $\tau \in [0,1]$ and $i= 0,\dots,N-1$. We can now concatenate these curves as follows
\begin{align}
\gamma_{\epi}^\delta(t)=(\gamma^\delta(t),\xi^\delta(t))=(\gamma^{\delta,i}(Nt-i),\xi^{\delta,i}(Nt-i)) \quad i/N \leq t < (i+1)/N
\, ,
\end{align} 
such that the curve $\gamma^\delta_{\epi} \in \Gamma_{\epi}$. It follows then that
\begin{align}
\max_{t \in [0,1]}\cG_I(\gamma^\delta_{\epi}) &\leq \sup_{t\in [0,1] }I(\gamma(t)) + \max_{i=0,\dots,N-1}\frac{\abs{\lambda}}{2}W_2^2(\gamma\bra{t_i},\gamma\bra{t_{i+1}}) \\
& \leq c+ \frac{\eps}{2} + \max_{i=0,\dots,N-1}\frac{\abs{\lambda}}{2}W_2^2(\gamma\bra{t_i},\gamma\bra{t_{i+1}}) \, .
\end{align}
Note that the curve $\gamma \in C([0,1];\cP_2(M))$ is uniformly continuous, since $[0,1]$ is compact. Thus, we can find a $\delta>0$ small enough such that
\begin{align}
\frac{\abs{\lambda}}{2}W_2^2(\gamma\bra{s},\gamma\bra{t}) \leq \frac{\eps}{2} \, ,
\end{align}
for all $\abs{s-t} \leq \delta$ and $ s,t \in [0,1]$. Thus, we have that
\begin{align}
\max_{t \in [0,1]}\cG_I(\gamma^\delta_{\epi}) \leq c+ \eps \, .
\end{align}
Using the inequality in~\eqref{eq:cubound}, it follows that
\begin{align}
c \leq \inf_{\gamma_{\epi} \in \Gamma_{\epi}} \max_{t \in[0,1]} \cG_I(\gamma_{\epi}(t)) \leq c+ \eps \, .
\end{align}
Since $\eps>0$ is arbitrary, it follows that
\begin{align}
\inf_{\gamma_{\epi} \in \Gamma_{\epi}} \max_{t \in[0,1]} \cG_I(\gamma_{\epi}(t))= c= \inf_{\gamma \in \Gamma} \sup_{t \in[0,1]} I(\gamma(t)) \, .
\end{align}
Also we have that $c_1=\max\set{I(\mu),I(\nu)}=\max \set{\cG_I(\mu,I(\mu)), \cG_I(\nu,I(\nu))}$ and from the above identity
that $\inf_{\gamma_{\epi} \in \Gamma_{\epi}} \max_{t \in[0,1]} \cG_I(\gamma_{\epi}(t)) = c >c_1$. Furthermore, if $I$ satisfies $\mathrm{(\hyperlink{MPS}{MPS})}$, it follows that $\cG_I$ satisfies it as well. Let $(\mu_n, \xi_n)$ be a Palais sequence.
Since $\abs{d \cG_I}(\mu_n,\xi_n) \to 0$ it follows from~\thref{lem:slopeaway} that for $n$ large enough the sequence must be of the form $(\mu_n,I(\mu_n))$ and that $\abs{d I}(\mu_n) \to 0$.  Since $\cG_I((\mu_n,I(\mu_n)))=I(\mu_n) \to c$, it follows that $\mu_n$ is a Palais sequence for $I$. Thus, we can construct a subsequence which converges to some $\mu^* \in \cP_2(M)$ and by extension
to $(\mu^*, c) \in \epi(I)$.
Finally we can apply~\thref{thm:mpc} to $\cG_I$ to extract the existence of a critical point $(\eta,c) \in \epi(I)$ such that $|d \cG_I|(\eta,c)=0$ and $\cG_I((\eta,c))=c= \inf_{\gamma_{\epi} \in \Gamma_{\epi}} \Upsilon(\gamma) $. However, the contraposition of \thref{lem:slopeaway} implies that $c=I(\eta)$ if $|d \cG_I|(\eta,c)<1$, from which it follows that $\abs{d I}(\eta)=\abs{d \cG_I}(\eta,I(\eta))=0$. Thus, $\eta$ is critical point of $I$ with critical value $c$. Also, since~$I$ is $\lambda$-geodesically convex it follows from~\thref{prop:equivslope} that $\abs{\partial I}(\eta)=0$.
 \end{proof}
\begin{rem}
We remark that a similar regularisation argument to the one used in the above proof, i.e. using the curves $\gamma^\delta$,  can also be used to prove that
\begin{align}
\inf_{\gamma \in \Gamma} \sup_{t \in [0,1]} I(\gamma(t)) = \inf_{\gamma \in \Gamma_{AC}} \max_{t \in [0,1]} I(\gamma(t)) \, ,
\end{align}
where $\Gamma_{AC}= \Gamma \cap AC([0,1]; \cP_2(M))$.
\end{rem}
\begin{rem}\thlabel{rem:lspace2}
Since all we have used in~\thref{def:lconvex},~\thref{lem:slopeaway},~\thref{prop:equivslope}, and~\thref{prop:connected}, is the existence of a unit speed minimising geodesic between two points~$\mu$ and~$\nu$, it follows that these results hold true if $M$ is replaced by $\cM$, a complete, separable, and locally compact length space. Furthermore, the abstract mountain pass theorem, i.e.~\thref{thm:mpwi}, continues to hold true if $\cP_2(M)$ is replaced by $\cP_2(\cM)$.
\end{rem}
\section{Application to the McKean--Vlasov model}\label{s:mva} 
This section is devoted to the analysis of the McKean--Vlasov free energy~$I$~\eqref{eq:I}.  As an immediate consequence of~\thref{thm:mpwi}, we obtain~\thref{intro:mpf}.
\begin{proof}[Proof of Theorem~\ref{intro:mpf}]
The functional $I:\cP(\T)\to \R$ is proper and l.s.c and since $\norm{D^2 W}_{\Leb^\infty(\T)} \leq C$ it is also $\lambda$-geodesically convex.  The space $\cP(\T)$ is compact and thus $I$ trivially satisfies~$\mathrm{(\hyperlink{MPS}{MPS})}$. Since $\mu_0$ is a strict local minimum of $I$, it follows that there exists an $R>0$, such that, for all $0<r<R$, $\mu_0$ is the unique minimiser of $I$ in $B^{W_2}_r(\mu_0)$. Thus, we have that $I(\mu)> I(\mu_0)$ for all $\mu \in \partial B^{W_2}_r(\mu_0)$, $0<r<R$. Since $\partial B^{W_2}_r(\mu)$ is compact (because $\cP(\T)$ equipped with the $W_2$ metric is compact) and $I$ is l.s.c, it follows that the minimum of $I(\mu)-I(\mu_0)$ must be attained by some $\mu_2^r \in \partial B^{W_2}_r(\mu_0)$. Thus, we have that
\begin{align}
I(\mu)-I(\mu_0) \geq I(\mu_2^r)-I(\mu_0) =: \delta(r)>0 \, , 
\end{align} 
for all $\mu \in \partial B^{W_2}_r(\mu_0)$ and all $0<r<R$. Let us set $r < \min\set{R, W_2(\mu_0,\mu_1)}$. Since any curve $\gamma \in \Gamma$ must pass through $\partial B^{W_2}_r(\mu_0)$, it follows that
\begin{equation*}
\inf_{\gamma \in \Gamma} \sup_{t \in [0,1]} I(\gamma(t)) \geq I(\mu_0) + \delta(r)= \max\set{I(\mu_0),I(\mu_1)} + \delta(r) \, . \qedhere
\end{equation*}
\end{proof}
At this level of generality, one still needs to find at least one strict local minimum to apply~\thref{intro:mpf}. Hence, as a next step, we would like to provide conditions on the interaction potential $W$ and the parameter values $\beta$ at which we can find two measures $\mu_0$ and $\mu_1$ which satisfy the assumptions of~\thref{intro:mpf}. We do this by using the results of~\cite{CGPS18} to argue that one can find potentials $W$ and parameter values $\beta_c$ such that the free energy $I$ has two distinct minimisers and one of them, $\mu^L$, is a strict local minimum of $I$. Such parameter values are referred to as discontinuous transition points of $I$ (cf.~\thref{defn:tp}).  In the second part, we formulate the large deviations results and complete the proof of~\thref{thm:ldi}. Let us recall some of the main definitions and results about the free energy functional from~\cite{chayes2010mckean} and~\cite{CGPS18}. 
\begin{defn}[Transition point] \thlabel{defn:tp}
A parameter value $\beta_c >0$ is said to be a \emph{transition point} of~$I$ from the uniform measure $\mu^L(\dx{x})=\dx{x}/L^d$ if it satisfies the following conditions:
\begin{enumerate}
\item For $0<\beta< \beta_c$, $\mu^L$ is the unique minimiser of $I$\,.
\item For $\beta=\beta_c$, $\mu^L$ is a minimiser of $I$\,.
\item For $\beta>\beta_c$, there exists $\mu_\beta \in \cP(\T)\setminus \set{\mu^L}$, such that $\mu_\beta$ is a minimiser of $I$\,.
\end{enumerate}
Additionally, a transition point  $\beta_c >0$ is said to be a \emph{continuous transition point} of $I$ if:
\begin{enumerate}
\item For $\beta=\beta_c$, $\mu^L$ is the unique minimiser of $I$\,.
\item Given any family of minimisers $\{\mu_\beta|\beta> \beta_c \}$, we have that
\begin{align}
\limsup_{\beta \downarrow \beta_c} \norm*{\mu_\beta-\mu^L}_{TV}=0 \, .
\end{align}
\end{enumerate}
A transition point $\beta_c$ which is not continuous is said to be \emph{discontinuous}.  
\end{defn}
In thermodynamics, continuous phase transitions correspond to second-order ones similar to those seen in the theory of magnetisation and spin systems~\cite{Dawson1983,Shiino1987,GomesPavliotis2018}, whereas discontinuous phase transitions correspond to first-order ones similar to the ones observed in nucleation processes or phase transformation from liquid to vapour~\cite{LP66}.

One can show that if $\beta_c$ is discontinuous, i.e. if either one of the  two conditions in~\thref{defn:tp} are violated, then it must be the case that condition (1) is violated (cf.~\thref{thm:summary}~\ref{p21}). This is the key idea we will use to obtain a set of conditions under which we can apply the result of~\thref{intro:mpf}. The original statement of these definitions and the proof of the above statement can be found in~\cite{chayes2010mckean}. We summarise the main results about the free energy functional in the theorem below. The proofs can be found in~\cite{chayes2010mckean} and~\cite[Theorems 5.11 and 5.19]{CGPS18}. The conditions are expressed in terms of the Fourier coefficients of $W$ denoted by]
\begin{equation}\label{e:FT:W}
\widehat{W}(k)=\intT{e_k(x) W(x)} \quad\text{ with }\quad e_k=L^{-d/2}\exp\bra*{\frac{2\pi \, i}{L} \; k \cdot x } \quad\text{ for }\quad k \in \Z^d \,.
\end{equation}
\begin{thm}\thlabel{thm:summary}
Assume that $W \in C^2(\T)$ and $\beta>0$.
\begin{tenumerate}
\item The free energy function $I : \cP(\T) \to \R \cup \set{+\infty}$ always has a minimiser $\mu \in \cP(\T)$ such that $\mu \ll \dx{x}$ with a positive and smooth density. 
\label{p1}
\item If there exists $k \in \Z^d\setminus\set{0}$ such that $\widehat{W}(k) <0$, then there exists a $\beta_c>0$ such that $\beta_c$ is a transition point of $I$. Furthermore, if the transition point $\beta_c>0$ is discontinuous, then there exist at least two minimisers of the free energy $I$ over $\cP(\T)$ at $\beta=\beta_c$, such that one is $\mu^L=L^{-d} \dx{x}$ and the other is some $\bar{\mu} \in \cP(\T)$. Additionally, if $\beta_c$ is discontinuous then $\beta_c <\frac{ L^{d/2} }{\abs*{\min_{k \in \Z^d, k\neq 0}\widehat{W}(k)}}$. \label{p21}
\item  \label{p2}
For $W$ with $\beta_c>0$ a transition point, let the set $K^\delta$ be given for any $\delta>0$ by
\begin{align}
K^\delta:=\set*{k \in \Z^d, k \neq 0: \widehat{W}(k) \leq \min_{k \in \Z^d, k\neq 0}\widehat{W}(k) +\delta} \, ,
\end{align}
Let $\delta_*>0$ be the smallest $\delta$ for which there
exist distinct $k^a,k^b,k^c \in K^{\delta_*}$ with $k^a=k^b+k^c$, if such points exist, else $\delta_*=\infty$. If $\delta_*$ is sufficiently small, then $\beta_c$ is a discontinuous transition point.
\item \label{p3}
Let $\set{W_n}_{n \in \N} \in C^2(\T)$, with $\beta_{c,n}>0$ the associated transition points, be a sequence of interaction potentials such that $\delta_* \to 0$ as $n \to \infty$. Assume there exists $N \in \N$ and a positive constant $C>0$ such that for all $n>N$, $\abs[\big]{\min_{k \in \Z^d, k\neq 0}\widehat{W}_n(k)} > C \delta_*^\gamma$ for any $\gamma< \frac{1}{2}$. Then 
for $n$ sufficiently large, $\beta_{c,n}$ is a discontinuous transition point and $\beta_{c,n}
<\frac{ L^{d/2} }{\abs*{\min_{k \in \Z^d, k\neq 0}\widehat{W}_n(k)}}$. 
\end{tenumerate}
\end{thm}
The above result provides conditions when to expect a discontinuous transition point. The case of the discontinuous transition point is particularly interesting for us as it implies the existence of a parameter value $\beta_c$ at which there are two 
distinct minimisers and hints at a possible scenario in which the mountain pass theorem
could be applied. To provide more intuition we show in the following lemma that any potential that under rescaling localises sufficiently fast but loses mass sufficiently
slow will eventually exhibit a discontinuous transition point for the associated free energy $I$.
\begin{lem}~\thlabel{cor:lm}
Let $W \in C^2(\T)$ be a compactly supported interaction potential with support strictly contained in $\T$ and $\intT{W}<0$. Assume further, that for some $\epsilon_1>0$ and all $\epsilon\in (0,\epsilon_1]$, it holds that
\begin{align}
L^{-d/2} \intT{ W(x) e^{i\frac{2 \pi \epsilon k \cdot x}{L}}} \geq L^{-d/2}\intT{W}:=-C \quad\text{ for all } k \in \Z^d \, .\label{eq:ftlb}
\end{align}
 Consider the rescaled potential, $W_\epsilon(x) = f(\epsilon) W(x/\epsilon)$ and positive function $f: (0,\epsilon_1] \to \R_+$.  If $\epsilon^{\ell} \lesssim f(\epsilon) \lesssim\epsilon^m$  as $\epsilon \to 0$ for $m >-d -2 $, $\ell \geq-d$, $\ell <\frac{m-d}{2} + 1$ (along with the natural restriction $\ell\geq m$), then for $\epsilon$ small enough, the associated free energy $I$ possesses a discontinuous transition point at some $\beta_c < \frac{ L^{d/2} } {\abs*{\min_{k \in \Z^d, k\neq 0}\widehat{W}_\epsilon(k)}}$.
\end{lem}
\begin{proof}
We proceed by checking that the conditions of~\thref{thm:summary}\ref{p3} hold for this class of potentials. We first check that for $\epsilon$ small enough, $W_{\epsilon}$ has at least one negative Fourier mode. Let $V:= \supp W$ and $V_\epsilon:=\supp W_\epsilon$. We have for
$k \in \Z^d$,
\begin{align}
\widehat{W}_\epsilon(k)& = L^{-d/2}\intT{W_\epsilon(x) e^{i\frac{2 \pi  k \cdot x}{L}}} \\
&= L^{-d/2}f(\epsilon)\int_{V^\epsilon} W(x/\epsilon) e^{i\frac{2 \pi  k \cdot x}{L}} \dx{x} \\
&= L^{-d/2}f(\epsilon) \epsilon^d \int_{V} W(x) e^{i\frac{2 \pi \epsilon  k \cdot x}{L}} \dx{x} \, \label{eq:est1} .
\end{align}
Since $ e^{i\frac{2 \pi \epsilon  k \cdot x}{L}}\to 1 $ uniformly  on $V$ as $\epsilon \to 0$, it follows that eventually $\widehat{W}_\epsilon(k)<0$
for $\epsilon$ sufficiently small since $\int_{V} W(x) e^{i\frac{2 \pi \epsilon  k \cdot x}{L}} \dx{x} \to \bra*{\intT{W}} <0 $. Using~\eqref{eq:ftlb}
and~\eqref{eq:est1}, we can now obtain the following bound
\begin{align}
\min\limits_{k \in \Z^d, k \neq 0} \widehat{W}_\epsilon(k) \geq -C f(\epsilon) \epsilon^d \, .
\end{align}
Since $W$ is even along every coordinate we have that
\begin{align} 
\intT{W(x) e^{i\frac{2 \pi \epsilon  k \cdot x}{L}}} &= \intT{W(x) \cos\bra*{\frac{2 \pi \epsilon  k \cdot x}{L}}} \\
&=\intT{W(x)\bra*{1 + \bra*{\frac{2 \pi |k|\epsilon x}{L}}^2 + O(\epsilon^4)}}
\end{align}
Fix some $k^* \in \Z^d$.
The above expansion tells us that we can find some $\epsilon_1$ sufficiently small and some $C_1>0$ independent of $\epsilon$ such that
\begin{align}
\widehat{W}_\epsilon(k^*) , \widehat{W}_\epsilon(2k^*)  \leq f(\epsilon) \epsilon^d (-C +C_1 \epsilon^2 ) \quad \text{ for all } \epsilon<\epsilon_1 \, .
\end{align}
  We can thus obtain the following bound
\begin{align}
-C f(\epsilon) \epsilon^d \leq \min\limits_{k \in \Z^d, k \neq 0} \widehat{W}_\epsilon(k) \leq f(\epsilon) \epsilon^d (-C +C_1 \epsilon^2 )
 \quad \text{ for all } \epsilon<\epsilon_1 \, . \label{eq:minb}
\end{align}
Combining the two of them we derive
\begin{align}
\begin{drcases}
\widehat{W}_\epsilon(k^*)-\min\limits_{k \in \Z^d, k \neq 0} \widehat{W}_\epsilon(k)
\leq C_1 f(\epsilon) \epsilon^{2+d}  \\
\widehat{W}_\epsilon(2k^*)-\min\limits_{k \in \Z^d, k \neq 0} \widehat{W}_\epsilon(k)
\leq C_1 f(\epsilon) \epsilon^{2+d} 
\end{drcases}
  \quad\text{ for all } \epsilon<\epsilon_1 \, ,
\end{align}
which tells us that $k^*,2k^* \in K^{C_1 f(\epsilon) \epsilon^{2+d}}$ and that $\delta_* \leq C_1 f(\epsilon) \epsilon^{2+d}$. Thus, $\delta_*  \lesssim  \epsilon^{m+d+2}$ and since $m >-d-2$, $\delta_* \to 0$ as $\epsilon \to 0$. Furthermore, using~\eqref{eq:minb} we can deduce
\begin{align}
\abs*{\min\limits_{k \in \Z^d, k \neq 0}\widehat{W}_\epsilon(k)} &\geq 
f(\epsilon) \epsilon^d (C - C_1 \epsilon^2 ) \geq C_2 \epsilon^{\ell +d}\, .
\end{align}
The fact that $\ell \geq -d$ tells us that
\begin{align}
 \abs*{\min\limits_{k \in \Z^d, k \neq 0} \widehat{W}_\epsilon(k)}\geq C_3 \delta_*^{\frac{\ell+d}{m+d+2}} \,.
\end{align}
We now use the assumption that $l <\frac{m-d}{2} +1 $ and apply~\thref{thm:summary}\ref{p3}, to obtain the desired result.
\end{proof}
The choice $f(\eps)= \eps^{-d}$ is an admissible scaling for the function $f$ in~\thref{cor:lm}.
Now that we have a set of concrete conditions under which we can expect there to be two distinct minimisers at a particular parameter value, we can try to apply the mountain pass theorem. To apply~\thref{thm:mpwi}, it is sufficient to show that $\mu^L$ is a strict local minima at parameter values $\beta_c$ and that this property is uniform, i.e.  we can find a ball $B_r^{W_2}(\mu^L)$ around $\mu^L$ in $W_2$ such that $I(\mu) \geq I(\mu^L) + \delta$ for all $\mu \in \partial B_r^{W_2}(\mu^L)$ and some $\delta>0$. 
In order to show this, we need the following comparison of $W_2$ with the homogeneous negative Sobolev space $\SobH^{-1}(\T)$, which we identify with all formal Fourier series of $\mu$ as defined in~\eqref{e:FT:W} given by $\sum_{k \in \Z^d \setminus \set{0}} \widehat\mu(k) e_k$ such that $\sum_{k \in \Z^d \setminus\set{0}}\frac{1}{|k|^2} \abs*{\widehat\mu(k)}^2 < \infty$. Note that the functions $\set{e_k}_{k \in \Z^d \setminus \set{0}}$ form an orthogonal basis for $\SobH^{-1}(\T)$ with respect to the inner product defined by duality with the homogeneous space $\SobH^{1}(\T)$. We have that $\skp{\mu,f}_{\SobH^{-1},\SobH^1}= \sum_{k \in \Z^d \setminus \set{0}} \widehat{\mu}(k) \widehat{f}(-k)$. Also, the inner product on $\SobH^1(\T)$ is defined as $\bra*{f,g}_{\SobH^1}=\sum_{k \in \Z^d \setminus\set{0}}|k|^2 \widehat{f}(k) \widehat{g}(-k)$. It is easy to check then that the Riesz representation of any $\mu \in \SobH^{-1}(\T)$ is given by $\sum_{k \in \Z^d \setminus{0}} \frac{1}{|k|^2}\widehat{\mu}(k) e_k \in \SobH^1(\T)$. 
\begin{lem}[Comparison of $\SobH^{-1}$ with $W_2$]\thlabel{lem:hwcomp}
Let $\mu_0, \mu_1 \in \cP(\T) \cap \Leb^\infty(\T)$. Then the following estimate
holds
\begin{align}
\norm{\mu_0 -\mu_1}_{\SobH^{-1}(\T)} \leq \bra*{\max \pra*{\norm{\mu_0}_{\Leb^\infty(\T)}, \norm{\mu_1}_{\Leb^\infty(\T)}}}^{1/2} W_2(\mu_0,\mu_1)
\end{align}
\end{lem}
\begin{proof}
The proof follows the argument in~\cite[Proposition 2.1]{loeper06}. Let $\mu(\cdot)$ be the optimal interpolant between $\mu_0$ and $\mu_1$ from~\thref{thm:monge}. Then by the Benamou--Brenier formulation of the optimal transport problem, there exists a vector field $[0,1]\ni t \mapsto v(t) \in \Leb^2(\mu(t); \R^d)$ such that the pair $(\mu(t),v(t))$ satisfies
\begin{align}
\partial_t \mu + \nabla \cdot(\mu v) =0 \,, \quad\text{for } t\in [0,1],
\end{align} 
in the sense of distributions. Now, we consider the sequence of parameterised problems given by
\begin{align}
\Delta \Psi_{t} =\mu(t) - L^{-d} \, \quad\text{for } t\in [0,1] .
\end{align}
Note that $\norm{\mu(t)}_{\Leb^\infty(\T)} \leq \max\set*{\norm{\mu_0}_{\Leb^\infty(\T)}, \norm{\mu_1}_{\Leb^\infty(\T)}}$~\cite[Corollary 17.19]{villani2008optimal}, and thus the above equation has  a unique weak solution in $\SobH^{1}(\T)$ for all $t \in [0,1]$. We know that
$\intT{\abs{v(t)}^2 \mu(t)}= W_2^2(\mu(t),\mu_0)=t^2 W_2^2(\mu_0,\mu_1) < \infty$. From this it follows that $\mu(t) v(t) \in \Leb^2(\T; \R^d)$ and thus $\nabla \cdot(\mu(t) v(t)) \in \SobH^{-1}(\T)$. Differentiating with respect to $t$ we have
\begin{align}
\Delta \partial_t \Psi_t = - \nabla \cdot(\mu(t) v(t)) \, .
\end{align}
It follows then that $\partial_t \Psi_t \in \SobH^1(\T)$. Multiplying by $\partial_t \Psi_t$ and integrating by parts with respect to the space variable and then integrating with respect to time, we obtain
\begin{align}
\norm{\nabla \Psi_1 -\nabla \Psi_0 }_{\Leb^2(\T)} &\leq \norm{\mu(t)}_{\infty}^{1/2}
W_2(\mu_0,\mu_1) \\
&\leq  \bra*{\max \pra*{\norm{\mu_0}_{\Leb^\infty(\T)}, \norm{\mu_1}_{\Leb^\infty(\T)}}}^{1/2} W_2(\mu_0,\mu_1)
\end{align} 
Since $\Psi_t$ is precisely the Riesz representation of $\mu(t)$ in $\SobH^{1}(\T)$, the claimed estimate holds.
\end{proof} 
\begin{rem}\thlabel{rem:hwcomp1}
For $d=1$, we remark that the $\SobH^{-1}(\mathbb{T})$-norm and $W_2(\cdot,\cdot)$-distance are comparable in both directions. Indeed, from the Kantorovich--Rubinstein dual formulation of the $W_1$ distance, we have that
\begin{align}
W_1(\mu,\nu)& = \sup_{\Lip(\varphi)\leq 1} \int_{\mathbb{T}} \varphi \dx{(\mu-\nu)} \\
& \leq L^{1/2}\sup_{ \norm{\varphi}_{\SobH^1(\mathbb{T})} \leq 1 } \int_{\mathbb{T}} \varphi \dx{(\mu-\nu)} =L^{1/2} \norm{\mu-\nu}_{\SobH^{-1}(\mathbb{T})} \, .
\end{align}
 Furthermore, since $\mathbb{T}$ is compact, we have that $W_2(\mu,\nu) \leq C \sqrt{W_1(\mu,\nu)}$. Thus, we have that
 \begin{align}
W_2(\mu,\nu) \leq C_1 \norm{\mu-\nu}_{\SobH^{-1}(\mathbb{T})}^{1/2} \, .
 \end{align}
 Furthermore, by using~\cite[Lemma 4.1 (ii) and (iii)]{MM13} we obtain that for $d=1$,
 \begin{align}
\norm{\mu-\nu}_{\SobH^{-1}(\mathbb{T})} \leq C_2 \sqrt{W_2(\mu,\nu)} \, .
 \end{align}
Thus, for $d=1$, the argument in the proof~\thref{lem:hwcomp} is not needed.
\end{rem}
The following lemma establishes the strictness of local minima in $W_2$ for discontinuous transition points.
\begin{lem}~\thlabel{lem:locmin}
Assume $W \in C^2(\T)$ with $\beta_c>0$ a discontinuous transition point. Then,  for $\beta \leq \beta_c$, the measure $\mu^L =L^{-d} \dx{x}$ is a strict local minimum of $I$. 
\end{lem}
\begin{proof}
By the definition of $\beta_c$ from~\thref{defn:tp}, we have that for $\beta\leq \beta_c$, $\mu^L$ is a minimiser of~$I$. The proof for $\beta<\beta_c$ is obvious. The idea of the proof is based on the fact that any minimiser of the free energy must be a solution of $T(\mu)=\mu- F(\mu)=0$ (cf.~\cite[Proposition 2.4]{CGPS18}), where $F : \SobH^{-1}(\T) \to \SobH^{-1}(\T)$ is the map given by
\begin{align}
F(\mu)= \exp\bra*{-\beta W \star \mu - \log \intT{\exp(-\beta W \star \mu)}} \, .
\end{align}
By the result of~\thref{thm:summary}\ref{p1}, all minimisers are smooth, and hence it is sufficient to consider the fixed point map $F$ on the space $\SobH^{-1}(\T)$. Note that here we identify a measure $\mu \in \cP(\T)$ with the the formal Fourier series $\sum_{k \in \Z^d \setminus \set{0}} \widehat{\mu}(k) e_k$.
It is possible to check now that, for $\beta \leq\beta_c$, $D T(\mu^L): \SobH^{-1}(\T) \to \SobH^{-1}(\T)$ is a bounded, linear, isomorphism. Indeed, we have that
\begin{align}
D T(\mu^L) (\eta) = \eta - \beta \mu^L \bra*{W \star \eta}  - \beta \mu^L \int_{\T} {W \star \eta}\dx{\mu^L}
\end{align} 
The above operator is bounded on $\SobH^{-1}(\T)$ since $W \in C^2(\T) \subset H^1(\T)$. 
Diagonalising $D T(\mu^L)$ using $\set{e_k}_{k \in \Z^d \setminus \set {0}}$, we obtain
\begin{align}
DT(\mu^L) e_k= \bra[\big]{1- \beta L^{-d/2} \widehat{W}(k)} e_{-k} \, .
\end{align}
It follows  that if $\beta \leq L^{d/2}/ \min_{k \in \Z^d \setminus \set{0}}\widehat{W}(k)$, then the above map is a bijection. That it is an injection is clear from the fact that if $DT(\mu^L) \eta_1 =DT(\mu^L) \eta_2$ then $\widehat{\eta_1}(k)=\widehat{\eta_2}(k)$ for all $k \in \Z^d \setminus \set{0}$. It is also surjective since for any $\eta \in \SobH^{-1}(\T)$, we have that $\sum_{k \in \Z^d \setminus \set{0}} \frac{\widehat{\eta}(k)}{1- \beta L^{-d/2} \widehat{W}(-k)}e_{-k} $ maps to $\eta$
under $D T(\mu^L)$. We know from~\thref{thm:summary}\ref{p21} that $\beta_c$ is lesser than this value and hence the result.  Now, by the inverse function theorem, there exists for some $\varepsilon>0$ an $\varepsilon$-open ball $B_{\varepsilon}^{\SobH^{-1}}(\mu^L)$ around $\mu^L$ in $\SobH^{-1}(\T)$ such that it is the unique solution of $T(\mu)=0$ in this ball. This tells us that $\mu^L$ is the unique minimiser of the free energy in $B_\varepsilon^{\SobH^{-1}}(\mu^L)$ at $\beta=\beta_c$. Note further that we have the following bounds for all $\mu \in B_{\varepsilon}^{\SobH^{-1}}(\mu^L)$
\begin{align}
\mu^L \exp\bra*{-2 \beta \norm{W}_{\SobH^{1}(\T)} \norm{\mu}_{\SobH^{-1}(\T)}}\leq F(\mu) \leq  \mu^L \exp\bra*{2 \beta \norm{W}_{\SobH^{1}(\T)} \norm{\mu}_{\SobH^{-1}(\T)}}  \, .
\end{align}
Additionally we have that $\norm{\mu-\mu^L}_{\SobH^{-1}(\T)} < \varepsilon$ from which it follows that
\begin{align}
\frac{\mu^L}{C} \leq F(\mu) \leq  C \, \mu^L  \qquad\text{with }\quad C:= \exp\bra*{2 \beta \norm{W}_{\SobH^{1}(\T)} (\norm{\mu^L}_{\SobH^{-1}(\T)} + \varepsilon)} \,,
\end{align}
for all $\mu \in B_\varepsilon^{\SobH^{-1}}(\mu^L)$.  Consider the set
\begin{align}
\cI:=\set*{\mu \in \SobH^{-1}(\T) \cap \cP(\T) \cap \Leb^{\infty}(\T) : \frac{\mu^L}{C} \leq \mu \leq C \mu^L } \,.
\end{align}
Then, any minimiser of $I$ must lie in $\cI$ by construction. Additionally, for all
$\mu \in \cI$ we have from \thref{lem:hwcomp} for some fixed constant $C_0=C_0(\mu^L,C)$ the bound
\begin{align}
\norm{\mu^L - \mu}_{\SobH^{-1}(\T)} \leq C_0 W_2(\mu^L,\mu) \, .
\end{align}
We can thus pick a ball $B_r^{W_2}(\mu^L)$ with $r>0$ sufficiently small such that $\norm{\mu-\mu^L}_{\SobH^{-1}(\T)}< \varepsilon$ for all $\mu \in B_r^{W_2}(\mu^L) \cap \cI $. Since all minimizers lie in $\cI$, it thus follows that we can find a ball in $W_2$ for which $\mu^L$ is the unique minimiser of $I$. Thus, $\mu^L$ is a strict local minima in $\cP(\T)$ equipped with the Wasserstein metric. 

The boundary of the ball $\partial B_r^{W_2}(\mu^L)$ is a compact set in $(\cP(\T),W_2)$, since $\T$ is compact. Hence, the l.s.c. functional $I$ attains its minimiser on this set, say $\mu^*$. Setting $\delta= I(\mu^*)-I(\mu^L)>0$ concludes the estimate in the lemma.
\end{proof}
We can now prove the existence of a mountain pass point in the presence of a discontinuous phase transition.
\begin{cor} \thlabel{cor:mpf}
Assume $W \in C^2(\T)$ with $\beta_c>0$ a discontinuous transition point, i.e.~there exist at least two distinct minimisers of $I$ at $\beta=\beta_c$ such that one is $\mu^L$ and the other is some $\bar{\mu} \in \cP(\T)$. It follows then that there exists a $\mu^* \in \cP(\T)$ distinct from $\mu^L$ and $\bar{\mu}$ such that $\abs{\partial I}(\mu^*)=\abs{d I}(\mu^*)=0$. Additionally,
$I(\mu^*)=c$ with 
\begin{align}
c= \inf\limits_{\gamma \in \Gamma} \sup_{t \in[0,1]} I(\gamma(t)) \, ,
\end{align}
where $\Gamma= \set{C([0,1]; \cP(\T)): \gamma(0)=\mu^L, \gamma(1)=\bar{\mu}}$.
\end{cor}
\begin{proof}
 We can directly apply~\thref{intro:mpf}, once we show that $\mu^L$ is a strict local minimum of $I$, which is established in~\thref{lem:locmin}. Thus, the result follows.
\end{proof}
\begin{rem}\thlabel{rem:whydc}
We have chosen to apply~\thref{intro:mpf} at a discontinuous transition point $\beta_c$ because we know from~\thref{lem:locmin} that $\mu^L$ is a strict local minimum. Furthermore, since $\beta_c>0$ is discontinuous, we know that $\min_{\cP(\T)}I(\mu)=I(\mu^L)=I(\bar{\mu})$. However, we expect, by the continuity of the minimum value of the free energy in $\beta$ (cf.~\cite[Proposition 2.4]{chayes2010mckean}), that the result of~\thref{intro:mpf} holds in a small neighbourhood of the critical value $\beta_c$. 

The situation is different at a continuous transition point $\beta_c>0$, where by~\thref{defn:tp}, the minimisers are unique at $\beta=\beta_c$. In this situation the mountain pass argument of~\thref{intro:mpf} can only be applied either by finding another measure $\mu \in D(I)$ (possibly a critical point) such the barrier value between $\mu^L$ and $\mu$ exceeds the maximum of their energies or by showing that for $\beta>\beta_c$ we can find measures satisfying the assumptions of~\thref{intro:mpf}. In principle, this may be possible for specific choices of $W$, but proving the existence of new critical points is usually non-constructive (cf.~\cite[Theorems 5.11 and 5.19]{CGPS18}). Thus, extracting any information about the value of their free energy is a challenging problem.
\end{rem}
We turn now, to the large deviations principle of the underlying particle system and the study of escape probabilities.
We start by stating, without proof, the reformulated version of the main result from~\cite{DG1987}. We also refer the reader to~\cite[Example 9.35, Section 13.3]{FengKurtz2006} for a discussion and proof of such large deviations principles for weakly interacting diffusions.
\begin{thm}~\thlabel{thm:dg}
Let $\cP^{(N)}(\T)$ be the space of empirical probability measures on $\T$, that is
\[
\cP^{(N)}(\T):= \set*{\mu \in \cP(\T): \mu= \frac{1}{N}\sum_{i=1}^N \delta_{x_i}, x_i \in \T} \, .
\] 
Assume that $\mu^{(N)}_0 \in \cP^{(N)}(\T)$ is such that there exists $\mu_0 \in \cP(\T)$ with $W_2(\mu^{(N)}_0,\mu_0) \to 0$ as $N \to \infty$.  Denote by $\cC_T$
the space $C([0,T]; \cP(\T))$, equipped with the topology of uniform convergence.
\begin{tenumerate}
\item For all open subsets $G$ of $\cC_T$ holds \label{dg1}
\[
\liminf_{N \to \infty} N^{-1} \log \mathbb{P}\bra*{\mu^{(N)}(\cdot) \in G , \mu^{(N)}(0)=\mu^{(N)}_0} \geq -\inf\limits_{\mu(\cdot) \in G, \mu(0)=\mu_0} S(\mu(\cdot)) \, .
\]
\item For all closed subsets $F$ of $\cC_T$ holds \label{dg2}
\[
\limsup_{N \to \infty} N^{-1} \log \mathbb{P}\bra*{\mu^{(N)}(\cdot) \in F , \mu^{(N)}(0)=\mu^{(N)}_0} \leq -\inf\limits_{\mu(\cdot) \in F, \mu(0)=\mu_0} S(\mu(\cdot)) \, ,
\]
\item For each compact subset $K$ of $\cP(\T)$ and $s \geq 0$ is the set
\[
\Phi_K(s)= \set*{\mu(\cdot) \in \cC_T: S(\mu(\cdot)) \leq s, \mu(0) \in K} \, ,
\]
compact.
\end{tenumerate}
Here $S: \cC_T \to \R \cup \set{+ \infty}$ is the action or rate functional 
given for $\mu \in AC^2([0,T]; \cP(\T))$ by
\begin{align}
S(\mu(\cdot)):= \frac{1}{4} \int_0^T \norm{\partial_t \mu - \nabla \cdot (\mu \nabla \bra{\beta^{-1} \log \mu +  W \star \mu})}_{\SobH^{-1}(\T,\mu)}^2 \dx{t} \, .
\end{align}
and by $+\infty$ otherwise.
\end{thm}
We are interested in using the above result to understand the probability 
of the empirical process escaping from the uniform state $\mu^L$ and reaching the clustered state $\bar{\mu}$ in time $T>0$. 
\begin{thm}\thlabel{ldfinal}
Assume $W \in C^2(\T)$ with $\beta_c$ a discontinuous transition point, i.e. there exist at least two distinct minimisers of $I$ at $\beta=\beta_c$ such that one is $\mu^L$ and the other is some $\bar{\mu} \in \cP(\T)$. It follows then that the underlying empirical process $\mu^{(N)} \in \cC_T$ 
with initial i.i.d uniformly distributed particles satisfies
\begin{align}
\mathbb{P}(\mu^{N}(T) \in \overline{B}_\eps^{W_2}\bra{\bar{\mu}}, \mu^{(N)}(0)=\mu^{(N)}_0) \leq \exp\bra*{-N \bra[\big]{\Delta - O(\eps^2))}- o_T(1)}
\end{align}
for $N$ sufficiently large, where $\overline{B}_\eps^{W_2}\bra{\bar{\mu}}$ is the closed ball of size $\eps>0$ around $\bar{\mu}$, $\Delta:= I(\mu^*)-I(\mu^L)$, where $\mu^*$ is the critical point defined in~\thref{cor:mpf}. Here, $o_T(1)$ is a constant which vanishes as $N\to \infty$ with a rate depending on the time interval $T>0$.
\end{thm} 
\begin{proof}
In order to prove this result we need to relate the rate functional $S$ with the energy functional $I$. We can assume without loss of generality that $\mu \in AC^2([0,T];H^1(\T) \cap \cP(\T))$ since $S(\mu)= + \infty$ otherwise. It follows that there exists $\phi \in \Leb^2([0,T]; \SobH^1(\T, \mu))$ \cite[Theorem 5.14]{santam} such that
\begin{align}
\partial_t \mu =  \nabla \cdot(\mu \nabla \phi) \, ,
\end{align}
where the above equation is satisfied in $\SobH^{-1}(\T,\mu)$. Thus, for $\mu \in AC^2([0,T];H^1(\T) \cap \cP(\T))$ we can rewrite the rate functional, using the chain rule for gradient flows discussed in~\cite[Section 10.1.2 E., Lemma 8.1.2]{ambrosio2008gradient} (see also~\cite{FengNguyen2012}), as follows
\begin{align}
S(\mu)&= \frac{1}{4}  \int_0^T \norm*{\partial_t \mu - \nabla \cdot (\mu  \nabla \bra{\beta_c^{-1}\log \mu +  W \star \mu})}_{\SobH^{-1}(\T,\mu)}^2 \dx{t} \\
&= \frac{1}{4}  \int_0^T \norm*{\phi-  \bra*{\beta_c^{-1}\log \mu + W \star \mu}}_{\SobH^{1}(\T,\mu)}^2 \dx{t}  \\
&= \frac{1}{4}  \int_0^T \norm*{\phi+  \bra*{\beta_c^{-1}\log \mu + W \star \mu}}_{\SobH^{1}(\T,\mu)}^2 \dx{t}  \\
&\qquad +\int_0^T \skp*{\beta_c^{-1}\log \mu + W \star \mu,\phi}_{\SobH^1(\T,\mu)} \dx{t}   \\
&=\frac{1}{4}  \int_0^T \norm*{\partial_t \mu + \nabla \cdot (\mu  \nabla \bra{\beta_c^{-1}\log \mu +  W \star \mu})}_{\SobH^{-1}(\T,\mu)}^2 \dx{t} \\
&\qquad +\int_0^T \skp[\big]{(\beta_c^{-1}\log \mu + W \star \mu),\partial_t \mu}_{\SobH^1(\T,\mu),\SobH^{-1}(\T,\mu)} \dx{t} \,.
\end{align}
We choose the closed subset $F=  \set{\mu \in \cC_T:\mu(T) \in \overline{B}_\eps^{W_2}\bra{\bar{\mu}}, \mu(0)=\mu^L}$ and we set $T^* = \arg \max_{t \in [0,T]} \bra[\big]{I(\mu(t))-I(\mu^L)}$ if it is uniquely defined or pick any one if it is not. We can then rewrite the rate functional as follows
\begin{align}
S(\mu)&= \frac{1}{4}   \int_0^T \norm*{\partial_t \mu + \nabla \cdot (\mu  \nabla \bra{\beta_c^{-1}\log \mu +  W \star \mu})}_{\SobH^{-1}(\T,\mu)}^2 \dx{t}v\\ 
&\qquad +\int_0^{T^*} \skp[\big]{(\beta_c^{-1}\log \mu + W \star \mu),\partial_t \mu}_{\SobH^1(\T,\mu),\SobH^{-1}(\T,\mu)} \dx{t}  \\
&\qquad + \frac{1}{4}  \int_{T^*}^T \norm*{\partial_t \mu - \nabla \cdot (\mu \nabla \bra{\beta_c^{-1}\log \mu +  W \star \mu})}_{\SobH^{-1}(\T,\mu)}^2 \dx{t} \\
& \geq  \int_0^{T^*} \skp[\big]{(\beta_c^{-1}\log \mu + W \star \mu),\partial_t \mu}_{\SobH^1(\T,\mu),\SobH^{-1}(\T,\mu)} \dx{t} \\
&= \max_{t \in [0,T]} (I(\mu(t))-I(\mu^L)) \, .
\end{align}
Note that we have again applied the chain rule for gradient flows from~\cite[Section 10.1.2 E]{ambrosio2008gradient}. The estimate implies the lower bound
\[ 
  \inf_{\mu \in F}S(\mu) \geq \inf_{\mu \in F \cap AC^2}\max_{t \in [0,T]} \bra[\big]{I(\mu(t))-I(\mu^L)} \,.
\]
At this point, we cannot apply \thref{intro:mpf} directly, since $F$ contains curves with varying endpoints not necessarily critical points. To handle this case, we define
\begin{align*}
F_{\epi}= \Bigl\{(\mu(\cdot),\xi(\cdot)) \in C([0,T]; \epi(I)): &(\mu(0), \xi(0))=(\mu^L,I(\mu^L)),\\
&(\mu(T), \xi(T)) \in \bigcup\nolimits_{\mu \in \overline{B}_\eps^{W_2}\bra{\bar{\mu}} }(\mu, I(\mu)) \Bigr\}
\end{align*}
If $\mu \in F \cap AC^2$, then the function $t \mapsto I(\mu(t))$ is absolutely continuous by \cite[Section 10.1.2 E.]{ambrosio2008gradient}. Thus, the curve $(\mu(\cdot),I(\mu(\cdot)))$ lies in the set $F_{\epi}$  with $\max_{t \in [0,T]} (I(\mu(t))-I(\mu^L))= \max_{t \in [0,T]} (\cG_I(\mu(t), I(\mu(t)))-\cG_I(\mu^L, I(\mu^L)))$. Thus, we have that
\[
\inf_{\mu \in F \cap AC^2}\max_{t \in [0,T]} \bra[\big]{I(\mu(t))-I(\mu^L)} \geq \inf_{\mu \in F_{\epi}}\max_{t \in [0,T]} \bra[\big]{\cG_I(\mu(t), \xi(t))-\cG_I(\mu^L, I(\mu^L)) }
\]
Now, we argue that if $\eps$ is small enough the above quantity can be made arbitrarily close to $\Delta$. For doing so, we define $\delta>0$ such that $\inf_{\mu \in F_{\epi}}\max_{t \in [0,T]} \bra[\big]{\cG_I(\mu(t), \xi(t))-\cG_I(\mu^L, I(\mu^L)) } = \Delta -2\delta$. Then, we find $(\tilde{\mu}(t), \xi(t)) \in F_{\epi}$ with $\max_{t \in [0,T]} \bra[\big]{\cG_I(\tilde{\mu}(t), \xi(t))-\cG_I(\mu^L, I(\mu^L)) }\leq \Delta-\delta$ from which it follows that $I(\tilde{\mu}(T))-I(\mu^L) \leq \Delta-\delta.$  
Let
\begin{align*} 
\Gamma_{\epi}=\Bigl\{ (\mu(\cdot),\xi(\cdot)) \in C([0,T];\epi(I)) : &(\mu(0),\xi(0)) = (\mu^L,I(\mu^L)), \\
&(\mu(T),\xi(T))=(\bar\mu, I(\bar\mu)) \Bigr\} \subset F_{\epi}\,
\end{align*}
We know that $\inf_{\mu \in \Gamma_{\epi}}\max_{t \in [0,T]} \bra[\big]{\cG_I(\mu(t), \xi(t))-\cG_I(\mu^L, I(\mu^L)) }=\Delta$. Thus, if we take any continuous curve $(\mu(s), \xi(s))$ in $\epi(I)$ from $(\tilde{\mu}(T), I(\tilde{\mu}(T)) )$ to $(\bar{\mu},I(\bar{\mu}))$ parametrised by $s \in [0,1]$, $\cG_I(\cdot,\cdot)$ must exceed or be equal to $I(\bar{\mu})+ \Delta$ at some $s \in [0,1]$. Indeed, if this would not be the case then we could concatenate $(\tilde{\mu}(t), \xi(t))$ and $(\mu(s),\xi(s))$ to obtain, after reparametrisation, a new curve $[0,1]\ni t \mapsto (\mu(t), \xi(t))$ in $\epi(I)$ from $(\mu^L,I(\mu^L))$ to $(\bar{\mu},I(\bar{\mu}))$ such that $\max_{t \in [0,1]} \cG_I(\mu(t), \xi(t)) < \Delta$, a contradiction, since this curve is also an element of~$\Gamma_{\epi}$. 

We pick the curve $(\mu(\cdot), I(\mu(\cdot))$ where $\mu \in C([0,1];\cP_2(M))$ is a unit speed minimizing geodesic between $\tilde{\mu}(T)$ and $\bar{\mu}$, as defined in~\thref{prop:geodesic}. Let $t'$ be the time at which $I(\mu(t'))$ exceeds $I(\bar{\mu}) + \Delta$. By $\lambda$-geodesic convexity of $I$ we have
\begin{align}
I(\bar{\mu}) + \Delta \leq I(\mu(t')) \leq  (1-t')I(\tilde{\mu}(T)) + t' I(\bar{\mu}) +\frac{\abs{\lambda}}{2}t'(1-t') \eps^2
\end{align}
Bounding $I(\tilde{\mu}(T))$ by $I(\bar{\mu}) + \Delta -\delta$, we obtain,
\begin{align}
I(\bar{\mu}) + \Delta &\leq I(\bar{\mu}) + (1- t')\Delta -(1-t')\delta +\frac{\abs{\lambda}}{2}t'(1-t') \eps^2 \\
& \leq I(\bar{\mu}) + \Delta -(1-t')\delta +\frac{\abs{\lambda}}{2}(1-t') \eps^2 \, .
\end{align}
From this it follows that
\[
\delta \leq \frac{\abs{\lambda}}{2} \eps^2 \, .
\]
Thus, we obtain
\[
\inf_{\mu \in F}S(\mu) \geq \inf_{\mu \in F_{\epi}}\max_{t \in [0,T]} \bra[\big]{\cG_I(\mu(t), \xi(t))-\cG_I(\mu^L, I(\mu^L)) } =\Delta-2 \delta  \geq \Delta - \abs{\lambda}\eps^2
\]
Finally, we can apply the result of~\thref{thm:dg}~\ref{dg2}, to obtain that
\begin{align}
\limsup_{N \to \infty} N^{-1} \log \mathbb{P}\bra*{\mu^{(N)}(\cdot) \in F , \mu^{(N)}(0)=\mu^{(N)}_0} &\leq -\inf\limits_{\mu(\cdot) \in F, \mu(0)=\mu^L} S(\mu(\cdot)) \leq  -\Delta + \abs{\lambda} \eps^2 \, ,
\end{align}
where we use that $W_2(\mu_0^{(N)}, \mu^L) \to 0$ as $N \to \infty$ is implied by the strong law of large numbers.
We set 
\begin{align}
a_N= N^{-1} \log \mathbb{P}\bra*{\mu^{(N)}(\cdot) \in F , \mu^{(N)}(0)=\mu^{(N)}_0} \, .
\end{align}
It follows that
\begin{align}
a_N \leq&\sup_{N_1 \geq N} a_{N_1} =  \bra*{\sup_{N_1 \geq N} a_{N_1} - \limsup_{N \to \infty} a_N}+\limsup_{N \to \infty} a_N 
 \leq C_{N,T} -\Delta + \abs{\lambda} \eps^2 \, ,
\end{align}
where $C_{N,T}=\bra*{\sup_{N_1 \geq N} a_{N-1} - \limsup_{N \to \infty} a_N}= o_T(1)$. Plugging in the expression for $a_N$, we obtain
\begin{align}
\mathbb{P}\bra*{\mu^{(N)}(\cdot) \in F , \mu^{(N)}(0)=\mu^{(N)}_0} &\leq e^{-N\bra[\big]{\Delta - \abs{\lambda}\eps^2} - o_T(1)} 
\end{align}
The result then follows from the above estimate and the definition of the set $F$.
\end{proof}
\begin{rem}
To obtain a result that is uniform in $T>0$ would require something stronger than the Dawson--G\"artner large deviations principle in~\thref{thm:dg}. The approach of quasi-potentials discussed in~\cite{DG1986,DG1989} may be the correct idea to use to obtain such a result. However, this would require much more information about the structure of the non-trivial minimiser and its basin of attraction. Since this is not the focus of this work, we refer to~\cite{Bashiri2020} for a first step in this direction for a particular choice of the interaction (and confining) potential. We hope to treat the general case in a future work.

Similarly, the $O(\eps^2)$ appearing in the exponent $\exp\bra[\big]{-N(\Delta- O(\eps^2))}$ can be removed if one can show that the minimiser $\bar{\mu}$ is a local basin of attraction for the McKean--Vlasov dynamics, i.e.  there exists some $\eps>0$ such that all measures in $\overline{B}_\eps^{W_2}(\bar{\mu})$ converge to $\bar{\mu}$
under the flow of the McKean--Vlasov PDE as $t \to \infty$. In this case we can choose the continuous curve between $\tilde{\mu}(T)$ and $\bar{\mu}$ (in the proof of~\thref{ldfinal}) to be the solution of the McKean--Vlasov PDE starting $\tilde{\mu}(T)$. This solution does not increase the energy and thus the $O(\eps^2)$ error from the $\lambda$-convexity argument will not appear in the exponent. Such a characterisation of $\bar{\mu}$ is expected under more specific assumptions on the potential~$W$.
\end{rem}
\subsection*{Acknowledgements}
The authors would like to thanks Jos\'e A. Carrillo and Greg Pavliotis for useful discussions during the course of this work. The authors would also like to thank the anonymous referees for their careful reading of the draft manuscript and their useful comments and suggestions.

\bibliographystyle{myalpha}
\bibliography{refs}

\end{document}